\documentclass[twoside,a4paper,reqno,11pt]{amsart}

\usepackage{amsfonts,amsmath,amssymb,amsthm}
\usepackage{amscd}
\usepackage{fullpage,enumitem,color,url}

\setlist[enumerate]{itemsep=3pt,topsep=3pt}
\setlist[enumerate,1]{label=\textup{(\roman*)}}
\setlist[enumerate,2]{label=\textup{(\alph*)}}

\newtheorem{lemma}{Lemma}[subsection]
\newtheorem{theorem}[lemma]{Theorem}

\newtheorem{corollary}[lemma]{Corollary}

\setlist[enumerate]{topsep=0pt}
\newcommand{\C}{\mathbb{C}}

\newcommand{\K}{\mathrm{K}}
\renewcommand{\H}{\mathbb{H}}
\renewcommand{\O}{\mathbb{O}}
\renewcommand{\P}{\mathbb{P}}
\newcommand{\Q}{\mathbb{Q}}
\newcommand{\R}{\mathbb{R}}

\newcommand{\Z}{\mathbb{Z}}

\newcommand{\CP}{\C\P}

\newcommand{\eps}{\epsilon}
\renewcommand{\o}{\otimes}

\newcommand{\SU}{\mathrm{SU}}

\newcommand{\coker}{\mathrm{coker}}

\newcommand{\Hom}{\mathrm{Hom}}

\newcommand{\Tor}{\mathrm{Tor}}

\newcommand{\U}{\mathrm{U}}

\newcommand{\Spin}{\mathrm{Spin}}
\newcommand{\HSpin}{\mathrm{HSpin}}

\newcommand{\Sp}{\mathrm{Sp}}

\newcommand{\SO}{\mathrm{SO}}

\newcommand{\ch}{\mathrm{ch}}
\newcommand{\ev}{\mathrm{ev}}

\newcommand{\Magma}{\textsc{Magma}}

\title{The classical topological invariants of homogeneous spaces}

\author{John Jones, Dmitriy Rumynin, Adam R. Thomas}

\begin{document}

\begin{abstract} 
We study the homogeneous spaces of a simply connected, compact, simple Lie group $G$ through the lens of $\K$-theory. Our general methods apply equally well to the case where $G$ is in one of the four infinite families of classical groups, or one of the five exceptional groups. In this paper we focus on the case of homogenous spaces $G/K$ where $G$ and $K$ have the same rank. The main examples we study in detail are the three symmetric spaces EIII, EVI, EVIII in Cartan's list of symmetric spaces. These are, respectively, homogeneous spaces for $E_6$, $E_7$, $E_8$ with dimensions $32$, $64$, $128$ and known as Rosenfeld projective planes.
\end{abstract}

\maketitle

\section*{Introduction} \label{sec:intro}

Let $G$ be a compact simply connected simple Lie group.  A (global) symmetric space for $G$ is a homogeneous space of the form $G/K$ where $K$ is the fixed point subgroup of an involution of $G$.  Cartan classified the symmetric spaces into seven infinite families, the symmetric spaces for the classical groups, together with twelve symmetric spaces for the exceptional groups.  A more conceptual, modern, classification is given in \cite{MR2785763} in terms of Freudenthal magic squares.  

The classical  symmetric spaces are very well understood.  For example, the papers of Borel and Hirzebruch \cite{MR0102800,MR0110105,MR0120664}, Bott and Samelson \cite{MR87035,MR0060233} give a wealth of useful information concerning the cohomology of the classical symmetric spaces, their homotopy groups, and the characteristic classes of the natural bundles over them.  In comparison the exceptional symmetric spaces are not at all well understood.  The paper by Piccinni \cite{MR3751970}, in particular Table A, gives an account of what is known and the relevant references.

Here is an example which shows quite how complicated this can be.  The simply connected compact Lie group $E_8$ has a $128$-dimensional symmetric space but there does not seem to be an explicit presentation of its rational cohomology algebra in the literature. To adapt a slogan from \cite{MR3544263}, we understand compact simply connected simple Lie groups only as far as we understand $E_8$.
 
Before going into more detail, we fix our conventions used throughout the paper. 
\begin{enumerate}
\item A Lie group is automatically compact and connected. All subgroups of Lie groups are assumed to be Lie groups in their own right.
\item Representations will always be finite-dimensional complex representations, unless explicitly stated otherwise, and $\K$-theory means complex $\K$-theory.  
\item Let $G$ be a Lie group. We make a choice of a maximal torus $T$ in $G$ and a base for the corresponding root system. Then we use the term {\em fundamental representation} of $G$ for an irreducible representation of $G$ whose highest weight is a fundamental weight for $T$. If $G$ is simply connected then there are $n$ fundamental representations $\rho_1, \dots, \rho_n$ of $G$ (which are all irreducible), where $n$ is the rank of $G$, and the representation ring of $G$ is the polynomial ring
$$
R(G) = \Z[\rho_1, \dots, \rho_n].
$$
For further details on the representation theory of Lie groups see, for example, \cite[Chapters~4, 6 and 7]{MR252560}, \cite[Chapters~13 and 14]{MR1249482}, \cite{MR1428422}, \cite{MR1153249}. 
\end{enumerate}

If $K \subset G$ is a subgroup of $G$ then the restriction homomorphism is a ring homomorphism $R(G) \to R(K)$ of representation rings.  Our primary strategy is to extract topological information about $G/K$ from this ring homomorphism.  We use $\K$-theory to study topological invariants and to express these invariants in terms of representation theory. In doing so, we take advantage of the usual natural operations, $\lambda$-operations $\lambda^k$ and Adams operations $\psi^k$, acting on these rings. Then we use the powerful methods of representation theory, both conceptual and computational, to get our results.  This leads to systematic methods that apply equally well to any simply connected Lie group. Our main objective is to apply these methods to understand more about the exceptional symmetric spaces. 

There is a significant computational aspect to this approach, with its emphasis on the exceptional symmetric spaces, since some big numbers occur in the representation theory of the exceptional Lie groups.  For example, the representation ring of $E_8$ is generated by the eight fundamental representations, with dimension
$$
3875, 147250, 669600, 6899079264, 146325270, 2450240, 30380, 248.
$$

The examples we study in detail in this paper are the three Rosenfeld projective planes, EIII, EVI, EVIII in Cartan's list, with dimensions 32, 64, and 128. In particular, our methods are tailored to the case where $G$ and $K$ have equal rank. We plan to study the other exceptional symmetric spaces, including those where $K$ has rank less than $G$, in detail in a follow up to this paper.

The following three results give us a good starting point for understanding topological invariants of a Lie group $G$, its homogeneous spaces $G/K$, and its classifying space $BG$, in terms of $\K$-theory and representation theory. 
\begin{enumerate}
\item
Let $G$ be a Lie group with $\pi_1(G)$ a free abelian group.  Then a theorem due to Luke Hodgkin \cite[p.~85]{MR0478156} gives a simple and conceptual description of the $\K$-theory of $G$ as a functor of the representation ring $R(G)$: precisely 
$$
\K^*(G) = \Tor^*_{R(G)}(\Z, \Z).
$$ 
\item 
Suppose that $K$ is a subgroup of $G$ with the same rank as $G$.  Then again we get a simple and conceptual description of the $K$-theory of $G/K$ using a theorem of Harsh Pittie \cite{MR0290402}: precisely 
$$
\K^*(G/K) = R(K)\o_{R(G)}\Z.
$$ 
\item
If $G$ is any Lie group, the well-known  Atiyah--Segal completion theorem \cite{MR0259946} gives the $\K$-theory of $BG$ as a functor of the representation ring of $G$: 
$$
\K^*(BG) = \hat{R}(G)
$$ 
where $\hat{R}(G)$ is the completion of $R(G)$ at the augmentation ideal.
\end{enumerate}

\subsection*{Contents}
We now give a summary of this paper and try to make the conceptual relations between the various, and quite diverse, sections clear.

\subsection*{Section \ref{sec:Ktheorybasics} }
Here we give precise formulations of the theorems of Hodgkin and Pittie and outline the proofs.  

\subsection*{Section \ref{sec:Adamsops}} 
Next, let $G$ be a Lie group with free abelian fundamental group.  We turn to the study of the Adams operations in $\K^*(G)$.  Here is an example which shows why this is important.  Representations of $K$ give vector bundles over $G/K$ and it is natural to use characteristic classes of these bundles to study the cohomology of $G/K$. For example, given a compact Hausdorff space $X$, the Chern character gives an isomorphism of $\Z/2$-graded rings
$$
\ch : \K^*(X)\o\Q \to  H^*(X; \Q).
$$
Here $H^*$ means the direct product of the $H^{k}$.  However, this does not determine the individual groups $H^k(X; \Q)$.  To do this we use Adams operations.  

According to Hodgkin's theorem the $\Z/2$-graded ring $\K^*(G)$ is generated by elements in $\K^1(G)$ and we must explain what we mean by Adams operations in the $\Z/2$-graded ring $\K^*(G)$.  We use the (standard) terminology $\Psi$-ring for a ring $R$ equipped with a sequence of ring endomorphisms $\psi^k$, $k \geq 1$ such that $\psi^1 = 1$ and $\psi^k\psi^l = \psi^{kl}$.  

The representation ring of a Lie group is an example of a $\Psi$-ring.  Following Bousfield \cite{MR1373816} we describe the notion of a $\Z/2$-graded $\Psi$-ring.  We then show how the $\Psi$-ring structure of $R(G)$ determines the $\Z/2$-graded $\Psi$-ring structure of $\K^*(G)$. 

\subsection*{Section \ref{sec:type}} 
In \cite{MR0059548} Serre proves that a Lie group of rank $n$ has the rational homotopy type of a product of $n$ odd dimensional spheres.  The type of $G$ is this list of odd integers.  The main theme in this section is to use a functorial version of the type of $G$.  

Let $V(G)$ be the indecomposable quotient of the ring $R(G)$. In other words,
$$
V(G) = R(G)/(\Z \oplus I^2(G)), \quad I^2(G) = I(G)^2
$$
where $I(G)$ is the kernel of the augmentation map
$$
R(G) \rightarrow \Z, \quad [\rho] \mapsto \dim (\rho).
$$
It is a free abelian group with rank equal to the rank of $G$. It also has Adams operations. More precisely, it is a $\Psi$-module in the usual sense. We show how to calculate the type of $G$ from the $\Psi$-module $V(G)$.

Next let $\pi(G) \o \Q$ be the direct sum of the nonzero rational homotopy groups of $G$.  We show that there is a natural isomorphism
$$
\pi(G)\otimes \Q \to V^*(G)\o\Q.
$$
where $V^*(G)$ is the dual of $V(G)$.  Using Adams operations we can refine this to give a natural description of the individual rational homotopy groups.   

The importance of this result is that it is functorial.  It gives a way of using representation theory to compute the homomorphism of rational homotopy groups defined by a homomorphism of Lie groups $h : H \to G$.  This turns out to be a very efficient way of doing such calculations. 

\subsection*{Section \ref{sec:Psimodule}} 

The previous section shows that the $\Psi$-module $V(G)$ is a very useful invariant of $G$ and in this section we study it in much more detail. 

Let $G$ be a simply connected simple Lie group, not isomorphic to one of the groups $\Spin(2n)$. In Section \ref{subsec:applications} we show that $V(G)\o\Q$ is generated as \ $\Psi$-module by one element. In Theorem \ref{thm:simplesinglegenpsimod} there is a list of simple Lie groups $G$ together with representations which define generators of the $\Psi$-module $V(G) \otimes \Q$. These results have two corollaries. The first is Corollary \ref{cor:Z2psi}, which shows that $\K^*(G)\o\Q$ is generated as a $\Z/2$-graded $\Psi$-ring by one element. The second is Corollary~\ref{cor:V(H)}, which states the following. 

Let $H$ be another Lie group and $f, g : H \to G$ be two  homomorphisms. Let $u \in R(G)$ be such that the element of $V(G)\o\Q$ defined by $u$ generates $V(G) \otimes \Q$ as a $\Psi$-module. Then
\[ f_*,g_* : \pi_*(H)\o \Q \to \pi_*(G)\o\Q \] 
are equal if and only if 
\[ f^*(u) = g^*(u) \in V(H).\]

In Section \ref{sec:spin2l} we study the $\Psi$-module $V(\Spin(2n)) \otimes \Q$ and adapt the two corollaries to give the corresponding results for the groups $\Spin(2n)$. 

\subsection*{Section \ref{sec:ratcoho}} 

Let $U$ be the vector bundle over $BG$ defined by a representation generating the $\Psi$-module $V(G)\o\Q$. In this section we apply the results of Section \ref{sec:Psimodule} to show that the rational cohomology of $BG$ is generated by the classes $\ch_q(U)$ with $q\geq 1$.  Here $\ch_q(U) \in H^{2q}(BG; \Q)$ is the degree $2q$ component of the Chern character of $U$.  This is obvious for the classical groups. But it tells us something very useful in the case of the exceptional groups.  More generally, let $\rho$ be a representation of $G$ and let $E_{\rho}$ be the vector bundle over $BG$ defined by $\rho$. In Theorem \ref{thm:chernofminisenough} we give a necessary and sufficient condition on $\rho$ which ensures that the elements $\ch_q(E_{\rho})$ with $q \geq 1$ generate the rational cohomology of $BG$.  

Now let $K$ be a subgroup of $G$, with the same rank as $G$. We use the previous result to distill a classical theorem of Borel into the following presentation of $H^*(G/K; \Q)$.  Let $\pi : BK \to BG$ be the map defined by the inclusion of $K$ in $G$.  We can choose $BK$ in such a way that $\pi$ is a fibre bundle with fibre $G/K$.  Let $j : G/K \to BK$ be the inclusion of a fibre.  The homomorphism $j^*: H^*(BK; \Q) \to H^*(G/K; \Q)$ is surjective since $K$ and $G$ have the same rank.  Then
 $$
 j^*(\ch_q(\pi^*(U)) = 0 \quad  \text{with $q \geq 1$}
 $$
since $j^*(\pi^*(U))$ is a trivial bundle.  We show that the kernel of $j^*$ is the ideal $I$ generated by $\pi^*(\ch_q(U))$ with $q \geq 1$.  Therefore $j^*$ defines an isomorphism
$$
H^*(BK; \Q)/I \to H^*(G/K; \Q).
$$
We can, of course, replace the $\ch_q$ by the Chern classes $c_q$. 

\subsection*{Section \ref{sec:ratelliptic}} 
This section contains a key idea. We can make a very clean transition from $\K$-theory and rational homotopy groups to rational cohomology by using the theory of {\em rationally elliptic spaces} due to Felix, Halperin, and Thomas \cite{felix2001rational}.  Any homogeneous space is a rationally elliptic space.  This section is short but important.  The theory tells us that once we know the rational homotopy groups of a rationally elliptic space we can read off a lot of structural information about its rational cohomology.  For example, in the case where the Euler characteristic of $\pi_*(X)\o\Q$ is zero, it gives a general formula for the Poincar\'{e} series and Euler characteristic of $H^*(X; \Q)$.  

\subsection*{Section \ref{sec:rathomRosenfeld}}
We now focus on the three largest Rosenfeld projective planes. These are the  homogeneous spaces $G/K$ where $G$ and $K$ are given in the following table.
\begin{center}
\begin{tabular}{ | c | c | c | c | c | c | }
\hline
$G$ & $K$ & dim. & (a) & (b) & (c) \\
\hline
$E_6$ & $\Spin(10)\times_{C_4} \U(1)$ & 32  & EIII & $P^2(\O \o \C)$  & $R_5$ \\
\hline
$E_7$ & $\Spin(12)\times_{C_2}\Sp(1)$ & 64  & EVI & $P^2(\O \o \H)$ & $R_6$ \\
\hline
$E_8$  & $\Spin(16)/C_2$ & 128 & EVIII & $P^2(\O\o\O)$ & $R_7$ \\
\hline
\end{tabular}
\end{center}
The last three columns are: (a) the label in Cartan's list of symmetric spaces, (b) the Rosenfeld notation, (c) the notation we use.
In \cite[Table~2]{MR2785763} they appear as the exceptional symmetric spaces of Grassmannian type.  It is usual to complete this to a list of seven, by setting $R_1 = P^2(\R)$, $R_2 = P^2(\C)$, $R_3 = P^2(\H)$ and $R_4 = P^2(\O)$. The projective planes over $\R, \C$ and $\H$ are extremely well-known and $R_4 = F_4 / \Spin(9)$ is given a full treatment in \cite[Chapter~5, Section~19]{MR0102800}. 

In this section we explain how to calculate the rational homotopy groups of these Rosenfeld projective planes and set out what the theory of rationally elliptic spaces tells us about the rational cohomology of these symmetric spaces.
  
\subsection*{Section \ref{sec:ratcohR567}}

We begin with a description of the classical way of calculating $H^*(G/K; \Q)$ when $G$ and $K$ have the same rank, as well as describing our approach to the same calculation. Then we use our method to obtain explicit presentations of the rational cohomology algebras of $R_5, R_6, R_7$. Our treatment is quite brief in the cases of $R_5$ and $R_6$. We state the results of our calculations and check that our presentations are consistent with those in the literature. The most complicated case is $R_7$ and since there does not seem to be any presentation of $H^*(R_7; \Q)$ in the literature we go into much more detail in this case. 

\subsection*{Section \ref{sec:reptheorycomps}}
Here we explain how we do the computations in representation theory required to prove Theorems \ref{thm:minmodgen} and \ref{thm:minmodgenSpineven} in Section \ref{sec:Psimodule} and Lemma \ref{lem:kernelsofhmaps} in Section \ref{sec:rathomRosenfeld}.

\subsection*{Computations}

Inevitably, the calculations in Sections \ref{sec:ratcohR567} and \ref{sec:reptheorycomps} require the use of a computational algebra package, and we use \Magma{}. The website \url{https://github.com/adamthomas22/LieGroupsRepTheoryCalcs} provides accompanying code for these calculations. In particular, it contains explicit \Magma{} files for each of the computations we do. 

\subsection*{Acknowledgments}
The authors thank the anonymous referee for their careful reading and many helpful suggestions. 

This research was funded in part by the EPSRC, EP/W000466/1 (Thomas). For the purpose of open access, the authors have applied a Creative Commons Attribution (CC BY) licence to any Author Accepted Manuscript version arising from this submission.

\subsection*{Statements and Declarations}

On behalf of all authors, the corresponding author states that there is no conflict of interest. No datasets were generated or analysed during the current study.

\section{The $\K$-theory of Lie groups and symmetric spaces} \label{sec:Ktheorybasics} 

\subsection{The $\K$-theory of Lie groups} 
Recall that for any connected compact Hausdorff space $X$ 
$$
\K^1(X) = [X, \U].
$$
Here $[ , ] $ means homotopy classes of maps and $\U = \U(\infty)$ is the direct limit of the unitary groups $
\U(n)$ where $\U(n)$ is embedded as $\U(n) \oplus 1$ in $\U(n+1)$.  

Now suppose $G$ is a Lie group and $\rho : G \to \U(n)$ is a representation of $G$. We can compose $\rho$ with the stabilisation homomorphism $\U(n) \to \U$ to get the element
$$
\beta(\rho)  \in \K^1(G).
$$
This gives a homomorphism of abelian groups 
\[ \beta : R(G) \to \K^1(G).\]

In \cite[Lemma~4.1]{MR214099} Hodgkin proves that the kernel of this homomorphism is
$$
\Z \oplus I^2(G) \subseteq R(G).
$$
As in the introduction write
$$
V(G) = R(G)/(\Z \oplus I^2(G))
$$
for the indecomposable quotient of the representation ring $R(G)$. It is a free abelian group of rank equal to the rank of $G$.  We will often identify $V(G)$ with a subgroup of $\K^1(G)$ using $\beta$. We can now state Hodgkin's theorem
\cite[Theorem~A]{MR214099}.

\begin{theorem} \label{thm:Hodgkinsthm} Let $G$ be a Lie group with $\pi_1(G)$ free abelian. Then
$$
\K^*(G) = \Lambda^*(V(G))
$$
is the $\Z/2$-graded exterior algebra over $\Z$ generated by $V(G) \subseteq \K^1(G)$.
\end{theorem}

The first step in Hodgkin's proof is to show that $\K^*(G)$ is torsion free. It follows from this that $\K^*(G)$ is a $\Z/2$-graded Hopf algebra over $\Z$, and this allows him to use facts about Hopf algebras to complete the proof.  In \cite{MR178092}, Atiyah gives a different way of completing the proof, assuming that $\K^*(G)$ is torsion free, by using the Chern character.

This result can be reformulated to say that 
$$
\K^*(G) = \Tor^*_{R(G)}(\Z, \Z),
$$
as $\Z/2$-graded rings.

\subsection{The $\K$-theory of homogenous spaces.} 
Let $G$ be a Lie group and let $K$ be a subgroup of $G$.  In \cite{AH} Atiyah and Hirzebruch describe a natural homomorphism
 $$
 \alpha : R(K) \to \K^0(G/K)
 $$
This homomorphism associates to a representation $\rho : K \to \U(n)$ the homogeneous vector bundle $V_{\rho}= G \times_{\rho} \C^n$.  It is clear that if the representation $\rho$ extends to $G$ then $V_{\rho}$ is isomorphic to the product bundle $G/K \times \C^n$.  In this way we get a homomorphism
$$
\alpha : R(K) \otimes_{R(G)} \Z \to \K^0(G/K).
$$
Atiyah and Hirzebruch conjectured that if $G$ and $K$ have the same rank then this ring homomorphism is surjective.
In \cite{MR0290402} (see also \cite{MR0372897}), Harsh Pittie  proves the following result.

\begin{theorem} \label{thm:Pittiesthm}
Let $G$ be a Lie group such that $\pi_1(G)$ is a free abelian group. Let $K$ be a subgroup of $G$ with the same rank as $G$. Then
$$
\alpha : R(K)\otimes_{R(G)} \Z \to \K^0(G/K) 
$$
is a ring isomorphism and
$$
\K^1(G/K) = 0.
$$
\end{theorem}
Pittie does not explicitly state that $\K^1(G/K) = 0$ but his method obviously proves it, as we explain in the next section. 

\subsection{On the proofs.}
The above results are both consequences of the K\"{u}nneth theorem, due to Luke Hodgkin \cite[Theorem~1, p.4]{MR0478156}, in Atiyah--Segal equivariant $\K$-theory. Let $G$ be a Lie group, such that $\pi_1(G)$ is free abelian.  Let $X$ and $Y$ be compact $G$-spaces that are locally contractible and of finite covering dimension. There is a spectral sequence with $E_2$ page
$$
E_2^{p, q} = \Tor^{p,q}_{R(G)}(\K_G^*(X), \K_G^*(Y)) \Longrightarrow \K_G^*(X \times Y).
$$
$$
d_r : E_r^{p, q} \to E_r^{p-r, q+r-1}
$$
Here $p \in \Z$ and $p\geq 0$ and $q \in \Z/2$ since $\K_G^*$ is $\Z/2$-graded.  

For example the $E_2$ page has two rows
$$
E_2^{p,0} = \Tor^{p}_{R(G)}(\K_G^0(X), \K_G^0(Y)) \oplus \Tor^{p}_{R(G)}(\K_G^1(X), \K_G^1(Y))
$$
$$
E_2^{p,1} = \Tor^{p}_{R(G)}(\K_G^0(X), \K_G^1(Y)) \oplus \Tor^{p}_{R(G)}(\K_G^1(X), \K_G^0(Y)).
$$
Hodgkin was only able to prove the convergence of the spectral sequence in special cases. The proof that it converges when $\pi_1(G)$ is a free abelian group was completed by McLeod in \cite{MR0557175}. 

For example, when $X = G/K$ and $Y = G$  we can use the isomorphisms
$$
\K_G^*(G/K) = R(K),  \quad \K_G^*(G) = \Z, \quad  \K_G^{*}(G/K \times G) = \K^*(G/K)
$$ 
to get a spectral sequence
$$
E_2^{p,q} = \Tor^{p,q}_{R(G)}(R(K), \Z) \Longrightarrow  \K^*(G/K).
$$
Minami proved that this spectral sequence collapses at the $E_2$ page in \cite[Theorem~2.1]{MR415646}.

A proof of Theorem \ref{thm:Hodgkinsthm} is given by taking $K$ to be the trivial subgroup obtaining a spectral sequence 
$$
E_2^{p,q} = \Tor^{p,q}_{R(G)}(\Z, \Z) \Longrightarrow  \K^*(G),
$$
which collapses at the $E_2$ page. This gives a proof of Hodgkin's theorem. 

\section{Adams operations in $\K^*(G)$} \label{sec:Adamsops}
Hodgkin's theorem tells us that if $G$ is a Lie group with $\pi_1(G)$ free abelian, the $\Z/2$-graded ring $\K^*(G)$ is generated by elements in $\K^1(G)$. In this section we explain what we mean by Adams operations in a $\Z/2$-graded ring, and extend Hodgkin's theorem by working out the Adams operations in $\K^*(G)$.

\subsection{$\Psi$-modules} \label{sub:_psi_modules}
By a $\Psi$-module, we mean an abelian group $A$ equipped with a sequence of endomorphisms $\psi^k$, $k \geq 1$ such that
$$
\psi^1 = 1, \quad \psi^k\psi^l = \psi^{kl}.
$$
Such operations are usually referred to as Adams operations.

We need some notation.  Let $A$ be a $\Psi$-module.  An element $a \in A$ has {\em weight $r$} (for some non-negative integer $r$) if $\psi^ka = k^ra$ for all $k$. Let $\Z(r)$ be the $\Psi$-module freely generated (as an abelian group) by a single element of weight $r$.  In this notation 
$$
\tilde{\K}^0(S^{2r}) = \Z(r)
$$
as $\Psi$-modules (where $\tilde{\K}^0(X)$ denotes the usual reduced K-theory of $X$). We also use the notation
$$
\Z(r_1, \dots, r_l) = \Z(r_1) \oplus \dots \oplus \Z(r_l)
$$$$
 \Q(r_1, \dots, r_l) = \Z(r_1, \dots, r_l)\o\Q.
 $$

Here is an example which is important for us.  We take for granted the usual theory of $\Lambda$-rings (see \cite{MR0244387}). Let $R$ be a $\Lambda$-ring with augmentation ideal $I$ and suppose $I^2 = 0$.  Then as a ring
$$
R = \Z \oplus I
$$
with product
$$
(n + x)(m+y) = nm + ny + mx
$$
where $n,m \in \Z$, and $x,y \in I$.  It follows that if $x, y \in I$, then
\begin{align*}
\lambda^k(x+y) = & \ \lambda^k(x) + \lambda^k(y),\\
\lambda^k(xy) = & \ \lambda^k(x)\lambda^k(y) = 0, \\
\lambda^k(\lambda^l(x)) = & \ (-1)^{(k+1)(l+1)}\lambda^{kl}(x).
\end{align*}

The inductive formula for the Adams operations
\[
\psi^{k}(x)
=
(-1)^{k+1}\,k\,\lambda^{k}(x)
\;+\;
\sum_{i=1}^{k-1}(-1)^{i-1}\,\lambda^{i}(x)\,\psi^{k-i}(x)
\]
shows that if $x \in I$, then 
$$
\psi^k(x) = (-1)^{k+1}k\lambda^k(x).
$$
 If $n \in \Z$ and $x \in I$ then $\psi^k(n+x) = n + \psi^k(x)$, and so, in general, it is not equal to $(-1)^{k+1}k\lambda^k(n + x)$. So the whole theory of $\Lambda$-operations in $R$ boils down to the fact that the $\lambda^k$-operations define a $\Psi$-module structure on $I$.

If we start with the $\Lambda$-ring $R(G)$ where $G$ is a Lie group we can form the $\Lambda$-ring 
$$
S(G) = R(G)/I^2(G).
$$
Then the $\Lambda$-ring structure is determined by the $\Psi$-module structure of the augmentation ideal $I(G)/I^2(G)$ of $S(G)$.  We now have two quotient homomorphisms
$$
R(G) \to S(G) \to V(G)
$$
where, as in the introduction,  $V(G) = R(G)/(\Z \oplus  I^2(G))$.  Now $\Z \oplus I^2(G)$ is sub-$\Psi$-module of $R(G)$ so $V(G)$ inherits a $\Psi$-module structure.  The first quotient is a map of $\Lambda$-rings and the second is a map of $\Psi$-modules.  It is straightforward to check that the composition
$$
I(G)/I^2(G) \to S(G) \to V(G)
$$
is an isomorphism of $\Psi$-modules.

\subsection{$\Z/2$-graded $\Psi$-rings}
By definition, a $\Psi$-ring is a ring equipped with ring endomorphisms $\psi^k$ such that $\psi^1 = 1$ and $\psi^k\psi^l = \psi^{kl}$.  Any $\Lambda$-ring becomes a $\Psi$-ring where the $\psi^k$ are the Adams operations associated to the $\lambda$-operations.   

In \cite{MR1373816} Bousfield gives a thorough account of the general theory of $\Z/2$-graded $\Lambda$-rings.  Here we give a self-contained account of the basic theory of $\Z/2$-graded $\Psi$-rings, which is adequate for our applications. This is a pared down version of the theory, but as Bousfield points out, if the ring is torsion free the Adams operations determine the $\lambda$-operations. 

A {\em $\Z/2$-graded $\Psi$-ring} is a $\Z/2$-graded ring $A^*$, which is strictly graded commutative, and is equipped with $\Psi$-module structures $\psi^k$ on $A^0$, and $\phi^k$ on $A^1$ with the following properties.
\begin{enumerate}
\item The operations $\psi^k$ make $A^0$ into a $\Psi$-ring.
\item For $x \in A^0$ and $y \in A^1$, $\phi^k(xy) = \psi^k(x)\phi^k(y)$.
\item For $x,y \in A^1, \psi^k(xy) = k\phi^k(x)\phi^k(y)$.
\end{enumerate}

Let $X$ be a compact Hausdorff space. Then $\K^0(X)$ is a $\Lambda$-ring and the corresponding Adams operations make $\K^0(X)$ into a $\Psi$-ring. We define operations $\phi^k$ on $\K^1(X)$ which give a $\Psi$-module structure on $\K^1(X)$.

The ring $\K^0(\Sigma X)$ is a $\Lambda$-ring with augmentation ideal $\tilde{\K}^0(\Sigma X)$.  If $x,y \in \tilde{\K}^0(\Sigma X)$ then  $xy = 0$.  So it follows that the operations $\lambda^k$ define a $\Psi$-module structure on $\tilde{\K}^0(\Sigma X)$. Now we use the standard natural isomorphism
$$
\K^1(X) \to \tilde{\K}^0(\Sigma X)
$$
to transfer these operations on $\tilde{\K}^0(\Sigma X)$ to $\K^1(X)$. Define 
$$
\phi^k  : \K^1(X) \to \K^1(X)
$$
to be the endomorphism, of abelian groups, defined by 
$$
(-1)^{k+1}\lambda^k : \tilde{\K}^0(\Sigma X) \to \tilde{\K}^0(\Sigma X).
$$
This defines a $\Psi$-module structure on $\K^1(X)$. The sign convention is the one needed to get the correct signs in the following lemma. 

\begin{lemma} \label{lem:Z2gradedPsiRing}
The operations $\psi^k$ on $\K^0(X)$ and $\phi^k$ on $\K^1(X)$ make $\K^*(X)$ into a $\Z/2$-graded $\Psi$-ring.
\end{lemma}

\begin{proof}
The proof is an exercise in the theory of products in $\K$-theory, see \cite[Chapter~2, Section~2.6]{atiyah2018k}. The basic product is the external tensor product
$$
\K^{0}(X) \otimes \K^0(Y) \to \K^0(X \times Y),
$$
where $Y$ is also a compact Hausdorff space. For this particular exercise it is best to use the reduced version 
$$
\tilde{\K}^{0}(X) \otimes \tilde{\K}^0(Y) \to \tilde{\K}^0(X \wedge Y).
$$
We assume that $X, Y$ are connected, and we have chosen base points $x_0 \in X$, $y_0 \in Y$.  By definition
$$
X\wedge Y = X\times Y /(X\times y_0 \cup x_0 \times Y).
$$
Next take $X = Y$ and use the reduced diagonal $\Delta : X \to X \wedge X$ to get the internal product
$$
\tilde{\K}^{0}(X) \otimes \tilde{\K}^0(X) \to \tilde{\K}^0(X)
$$
Finally replace $X$ by $S^p\wedge X$ and $Y$ by $S^q\wedge X$ and use the appropriate suspension of $\Delta$ to get the product 
$$
\tilde{\K^0}(S^p \wedge X) \otimes \tilde{\K}^0(S^q\wedge X) \to \tilde{\K}^0(S^p \wedge X\wedge S^q \wedge X) \to \tilde{\K}^0(S^{p+q}\wedge X).
$$
We have to interchange $X$ and $S^q$, this introduces the sign $(-1)^q$, and use the standard identification of $S^p\wedge S^q$ with $S^{p+q}$.

Now use Bott periodicity to identify $\tilde{\K}^0(S^p \wedge X)$ with $\tilde{\K}^0( X)$ if $p$ is even, or with $\K^1(X)$ if $p$ is odd.   We end up with products  
$$
\K^1(X) \otimes\tilde{\K}^0(X)  \to \tilde{\K}^0(X), \quad \K^1(X) \otimes \K^1(X) \to \tilde{\K}^0(X).
$$
The verification of the formulas in the above definition follow from the fact that the Adams operations are ring homomorphisms of $\K^0(X)$, the definition of the operations $\phi^k$ in $\K^1(X)$ and the relation between the Adams operations and Bott periodicity.
\end{proof}

\subsection{The $\Z/2$-graded $\Psi$-ring $\K^*(G)$}
As usual, we assume that  $G$ is a Lie group with $\pi_1(G)$ free abelian.  

\begin{theorem} \label{thm:Z2gradedringisos}
The $\Z/2$-graded $\Psi$-ring $\K^*(G)$ is isomorphic to $\Lambda(V(G))$, the $\Z/2$-graded exterior algebra over $\Z$, generated by $V(G) \subseteq \K^1(G)$ equipped with the unique $\Z/2$-graded $\Psi$-ring structure determined by the $\Psi$-module structure on $V(G)$.  
\end{theorem}

The proof of Theorem \ref{thm:Z2gradedringisos} follows directly from the following lemma. Recall that both $R(G)$ and $\K^1(G)$ are $\Psi$-modules. 
\begin{lemma} \label{lem:homofpsimodules}
The homomorphism $\beta: R(G) \to \K^1(G) = \tilde{\K}^0(\Sigma G)$ is a homomorphism of $\Psi$-modules.
\end{lemma}

\begin{proof} 
The isomorphism $\K^1(X) \to \tilde{\K}^0(\Sigma X)$ is defined by the the so-called clutching construction in $\K$-theory.  Given a compact Hausdorff space $X$ we take two copies $C_\pm(X)$ of the cone on $X$ and glue them together along $X$ to get
$$
\Sigma X = C_+(X ) \cup_X C_-(X).
$$
Suppose we have a map $f : X \to U(n)$. This gives us an isomorphism
$$
\hat{f} : X \times \C^n \to X \times \C^n \, , \quad (x,v) \mapsto (x, f(x)(v))
$$
so that we can form the vector bundle
$$
V(f) = C_+(X ) \times \C^n \cup_{\hat{f}}  C_-(X) \times \C^n
$$
over $\Sigma X$.  If we replace $f$ by $\lambda^k(f):  X  \to \U(\lambda^k(\C^n))$  it is clear that
$$
\lambda^k(V(f)) = V(\lambda^k(f)).
$$
It is straightforward to follow through the arguments of \cite[Chapter~3, Section~3.1]{atiyah2018k} in this special case to convert this observation into a proof of the lemma.
\end{proof}

In the case of a Lie group $G$, the $\alpha, \beta$ constructions of Section 1 are related in the following way.  The join $G*G$ has a standard free action of $G$ and $(G*G)/G$ is homotopy equivalent to $\Sigma G$, the suspension of $G$.  This gives us a map $\alpha : R(G) \to \K^0((G*G)/G) = \K^0(\Sigma G)$.  This map is clearly a map of $\Lambda$-rings.  On the other hand we have the map $\beta : R(G) \to \K^1(G)$.  Now we can compare these two maps by using the isomorphism $\sigma : \K^1(G) \to \tilde{\K}^0(\Sigma G)$ defined by the clutching construction. It is not difficult to show that the following diagram commutes
$$
\begin{CD}
R(G) @>{\alpha}>> \K^0(\Sigma G) \\
@V{\beta}VV @VV{p}V \\
\K^1(G) @>>{\sigma}> \tilde{\K}^0(\Sigma G)
\end{CD}
$$
where $p : \K^0(\Sigma G) \to \tilde{\K}^0(\Sigma G)$ is the projection. This gives a different proof of Lemma \ref{lem:homofpsimodules} in this special case.

\section{On the rational homotopy of $G$ and $G/K$} \label{sec:type}
In this section we make preliminary observations about the rational homotopy groups of Lie groups and homogeneous spaces.

\subsection{The type of a Lie group} 
Let $G$ be a Lie group of rank $n$.  Serre proves in \cite[Chapitre~V, Section~3]{MR0059548} that $G$ is rationally homotopy equivalent to a product of odd dimensional spheres, $S^{2r_1 - 1} \times \dots \times S^{2r_n - 1}$,
where
\begin{enumerate}
\item $1 \leq r_1 \leq r_2 \leq \dots \leq r_n$,
\item $\dim G = (2r_1-1) + \dots + (2r_n - 1)$.
\end{enumerate}
The {\em type} of $G$ is the sequence
$$
(2r_1-1,  \dots, 2r_n-1).
$$

In particular, for simply connected simple Lie groups the types are as follows (see \cite[p.11, p.14]{MR45129} for example): 
\begin{align*}
\SU(n):  \quad &  (3,5,7, \dots , 2n-1),\\
\Sp(n): \quad & (3,7,11, \dots , 4n-1),\\
\Spin(2n+1): \quad & (3,7,11, \dots , 4n-1),\\
\Spin(2n): \quad & (3,7,11, \dots , 4n-5, 2n-1),\\
G_2: \quad & (3,11),\\
F_4: \quad & (3,11,15,23),\\
E_6: \quad & (3,9,11,15,17,23),\\
E_7: \quad & (3,11,15,19,23,27,35),\\
E_8: \quad & (3,15,23,27,35,39,47,59).
\end{align*}
Using Serre's result, it is straightforward to calculate the $\Z/2$-graded $\Psi$-ring $\K^*(G)\o\Q$ and the $\Psi$-module $V(G)\o \Q$. 

\begin{theorem} \label{thm:isoswithtype}
Let $G$ be a Lie group of rank $n$ with $\pi_1(G)$ free abelian and let $(2r_1-1, \dots, 2r_n-1)$ be the type of $G$. 
There is an isomorphism of $\Psi$-rings 
\begin{enumerate}
\item \[ \K^*(G)\o \Q \to \Lambda_{\Q}(a_1, \dots, a_n),\]
where the weight of $a_i$ is $r_i$, and an isomorphism of $\Psi$-modules 
\item \[ V(G)\o \Q \to \Q(r_1, \dots, r_n).\]
\end{enumerate}
\end{theorem}
\begin{proof}
We know that $\K^1(S^{2r-1}) = \Z(a)$ where $a$ has weight $r$. Using the fact that $G$ is rationally homotopy equivalent to a product of odd-dimensional spheres, the result now follows directly from the K\"unneth formula for $\K$-theory.
\end{proof}

This tells us how to compute the type of $G$ from the representation ring of $G$: we must compute the eigenvalues of the Adams operations on $V(G)$.

\subsection{$V(G)$ and the rational homotopy groups of $G$}
If $X$ is a compact Hausdorff space and $f : S^n \to X$ is a map then we have the induced map
\[ f^* : \K^*(X) \to \K^*(S^n). \]
This clearly defines a homomorphism
\[ d : \pi_n(X) \to \Hom_{\Psi}(\K^*(X), \K^*(S^n)), \]
which is the Adams $d$-invariant.

We now examine the $d$-invariant
\[
d : \pi_{2r-1}(G) \otimes \Q \to \Hom_{\Psi}(\K^1(G), \Q(r)).
\]
If $f:S^{2r-1} \to G$ is a map, then $f^* : \K^1(G) \to \K^1(S^{2r-1})$ is uniquely determined by its restriction to the generators $V(G) \subseteq \K^1(G)$.  So we abuse notation and regard $d(f)$ as an element
$$
d(f) \in \Hom_{\Psi}(V(G), \Q(r)). 
$$
Finally, we rename $\Hom_{\Psi}(V(G), \Q(r))$ as $V^*_r(G)\o\Q$ since it is the weight $r$ subspace of the $\Psi$-module $V^*(G)\o \Q$ dual to the $\Psi$-module $V(G)$.  

\begin{theorem} \label{thm:dmapisiso}
Let $G$ be a Lie group such that $\pi_1(G)$ is torsion free. The homomorphism
$$
d : \pi_{2r-1}(G) \otimes \Q \to V^*_r(G)\otimes\Q.
$$
is an isomorphism.
\end{theorem}
\begin{proof}
If we replace $G$ by an odd sphere $S^{2r-1}$ this result is straightforward.  The full result follows because $G$ is rationally homotopy equivalent to a product of odd dimensional spheres.
\end{proof}
Suppose $h : H \to G$ is a homomorphism of Lie groups. Then Theorem \ref{thm:dmapisiso} gives a general method of computing the homomorphism $h_* : \pi_{2r-1}(H) \otimes \Q \to \pi_{2r-1}(G) \otimes \Q$ in terms of representation theory.  Explicitly, the following diagram commutes
$$
\begin{CD}
\pi_{2r-1}(H) \otimes \Q @>{h_*}>> \pi_{2r-1}(G) \otimes \Q \\
@V{d}VV @VV{d}V \\
V^*_r(H)\otimes\Q @>>{h_*}> V^*_r(G) \otimes \Q
\end{CD}
$$
where the horizontal arrows are the linear maps induced by $h$ and the vertical ones are the isomorphisms in Theorem \ref{thm:dmapisiso}.

\subsection{The rational homotopy groups of a homogeneous space}
Let $G$ be a Lie group.  It follows from Serre's theorem that $\pi_p(G)\o\Q = 0$ if $p$ is even.  Now suppose $K$ is a subgroup of $G$. It follows that the exact sequence in rational homotopy groups of the principal $K$-bundle $G \to G/K$ splits up into exact sequences
$$
0 \to \pi_{2r}(G/K)\o\Q\to \pi_{2r-1}(K)\o\Q \to \pi_{2r-1}(G)\o\Q \to \pi_{2r-1}(G/K)\o\Q \to 0
$$
for $r \geq 1$.  We refer to these exact sequences as the {\em four-term exact sequences}.  Using Theorem~\ref{thm:dmapisiso} we can rewrite these exact sequences as
$$
0 \to \pi_{2r}(G/K)\o \Q \to V^*(K)_r\o\Q \to V^*(G)_r \o \Q\to \pi_{2r-1}(G/K)\o\Q \to 0.
$$
This proves the following result.

\begin{theorem} \label{thm:psimodulehomsforinclusions}
Let $G$ be a Lie group with a subgroup $K$. Let $i : K \to G$ be the inclusion of $K$ and let
$$
i_* : V^*(K) \o \Q \; \to \; V^*(G)\o\Q
$$
be the homomorphism of $\Psi$-modules induced by $i$.  Then there are natural isomorphisms
$$
\pi_{2r}(G/K)\o \Q \; \longrightarrow \; \ker \, \Big(i^* : V^*(K)_{2r-1}\o\Q \to V^*(G)_{2r-1}\o\Q \Big),
$$
$$
\pi_{2r-1}(G/K)\o \Q \; \longrightarrow \; \coker \, \Big(i^* : V^*(K)_{2r-1}\o\Q \to V^*(G)_{2r-1}\o\Q \Big).
$$
\end{theorem}

\section{The $\Psi$-module structure of $V(G)$} \label{sec:Psimodule}
In this section we study the $\Psi$-module $V(G)$ in detail and give two applications to the $K$-theory and rational homotopy groups of $G$.  We use the notation and terminology for $\Psi$-modules described in Section 2. We say that a $\Psi$-module is cyclic if it is generated, as a $\Psi$-module, by one element.

\subsection{A basic lemma}
\begin{lemma} \label{lem:cyclicpsimodules}
\begin{enumerate}
\item The $\Psi$-module $\Q(r_1, \dots, r_l)$ is a cyclic $\Psi$-module if and only if the $r_i$ are distinct.
\item Suppose $r_1, \dots, r_l$ are distinct. Then $u$ is a generator of the $\Psi$-module $\Q(r_1, \dots, r_l)$ if and only if $u = u_1 + \dots + u_l$ where $u_i$ is a non-zero element of $\Q(r_1, \dots, r_l)$ with weight $r_i$.
\end{enumerate}
\end{lemma}
\begin{proof} 
Let $r$ be one of the $r_i$ and let $W_r$ be the weight $r$ subspace of $\Q(r_1, \dots, r_l)$.  So the dimension of $W_r$ is the number of $r_i$ equal to $r$.  Let $p_r : \Q(r_1, \dots, r_l) \to W_r$ be the standard projection onto $W_r$.  Suppose $\Q(r_1, \dots, r_l)$ is a cyclic $\Psi$-module. Then, since $p_r$ is a map of $\Psi$-modules $W_r$ is also a cyclic $\Psi$-module.  Let $w$ be a non-zero element of $W_r$.  Then $\psi^kw = k^rw$ and so the $\Psi$-module generated by $w$ must be $1$-dimensional and the dimension of $W_r$ must be one.  It follows that the $r_1, \dots, r_l$ must be distinct.

Recall that for a linear map $A : V \to V$, a vector $v \in V$ is called a cyclic vector for $A$ if 
$$
v, A(v), \dots, A^{l-1}(v)
$$
is a basis for $V$. 
 
The eigenvalues of $\psi^k$ are $k^{r_1}, \dots, k^{r_l}$.  If $r_1, \dots, r_l$ are distinct
and $k\geq 2$, 
then clearly 
there exists a cyclic vector for $\psi^k$. If $u$ is such a cyclic vector, it follows that
$$
u, \ \psi^k(u), \ \psi^{k^2}(u), \ldots, \ \psi^{k^{l-1}}(u)
$$ 
are a basis for $\Q(r_1, \dots, r_l)$ and so $\Q(r_1, \dots, r_l)$ is generated as a $\Psi$-module by $u$.

This completes the proof of the first part of the theorem. The second part follows directly from the above characterisation of cyclic vectors.
\end{proof}

\subsection{Applications} \label{subsec:applications}

We assume that $G$ is a simply connected simple Lie group that is not isomorphic to $\Spin(2n)$. The groups $\Spin(2n)$ are dealt with in Section \ref{sec:spin2l}. 

\begin{theorem} \label{thm:simplesinglegenpsimod}
The $\Psi$-module $V(G)\o\Q$ is cyclic.
\end{theorem}

\begin{proof}
Theorem~\ref{thm:isoswithtype} shows that the $\Psi$-module $V(G) \o \Q$ is $\Q(r_1, \dots, r_n)$ where the type of $G$ is $(2r_1-1, \dots, 2r_n-1)$. These integers are distinct (since we have excluded, in particular, the groups $\Spin(4k)$). The result now follows from Lemma \ref{lem:cyclicpsimodules}. 
\end{proof}

\begin{corollary} \label{cor:Z2psi}
The $\Z/2$-graded $\Psi$-ring $\K^*(G) \o\Q$ is generated by one element.
\end{corollary}

\begin{proof}
This follows immediately from Theorem \ref{thm:simplesinglegenpsimod}.
\end{proof}

\begin{corollary} \label{cor:V(H)}
Let $H$ be a Lie group and suppose $f, g : H \to G$ are homomorphisms of Lie groups.  Let $u \in R(G)$ be a representation whose image generates $V(G)\o \Q$ as a $\Psi$-module $V(G)\o\Q$. Then the two homomorphisms
$$
f_*, g_* : \pi_*(H)\o \Q \to \pi_*(G)\o \Q
$$
are equal if and only if 
$$
f^*(u) = g^*(u) \in V(H) \o \Q.
$$
\end{corollary} 

\begin{proof}
Since $V(G)\o\Q$ is generated as a $\Psi$-module by the element defined by $u$ it follows that if $f^*(u) = g^*(u) \in V(G) \o \Q$ if and only if the homomorphisms $f^*, g^* : V(G)\o \Q \to V(H)\o \Q$ are equal. The result now follows directly from Theorem \ref{thm:dmapisiso}. \end{proof}

The following result gives an explicit, `convenient' choice for a representation $u$ whose image in $V(G) \otimes \Q$ is a $\Psi$-module generator.  

\begin{theorem} \label{thm:minmodgen}
The following representations $u$ define generators of the $\Psi$-module $V(G)\o\Q$. 
\begin{enumerate}
\item $G = \SU(n)$ and $u$ is the representation afforded by the usual action of $G$ on $\C^{n}$. 
\item $G = \Sp(n)$ and $u$ is the representation afforded by the usual action of $G$ on $\H^n = \C^{n} \oplus \C^{n}$. 
\item $G = \Spin(2n+1)$ and $u$ is the representation afforded by the usual action of $G$ on $\C^{2n+1}$.
\item $G=G_2$ and $u$ is the $7$-dimensional fundamental representation. 
\item $G=F_4$ and $u$ is the $26$-dimensional fundamental representation. 
\item $G=E_6$ and $u$ is one of the two $27$-dimensional fundamental representations. 
\item $G=E_7$ and $u$ is the $56$-dimensional fundamental representation of $G$. 
\item $G=E_8$ and $u$ is the $248$-dimensional adjoint representation of $G$. 
\end{enumerate}
\end{theorem}

The proof of this theorem is postponed to Section \ref{sec:reptheorycomps}.

\subsection{The groups $\Spin(2n)$} \label{sec:spin2l}

We think the most useful result for these groups is the following theorem. Once more, the proof is postponed to Section \ref{sec:reptheorycomps}.

\begin{theorem} \label{thm:minmodgenSpineven}
The elements of $V(\Spin(2n))\o\Q$ defined by the vector representation and one of the spin representations, say $\Delta^+$, of dimension $2^{n-1}$, generate the $\Psi$-module $V(\Spin(2n))\o\Q$. 
\end{theorem}

The groups $\Spin(4k)$ have a repeated integer in their type and so $V(\Spin(4k))\o\Q$ is not cyclic. Therefore, the result is best possible in this case. 

The $\Psi$-module $V(\Spin(4k+2))\o\Q$ is cyclic but there does not seem to be any particularly `convenient' choice of a generator. So it might (or might not) be best to work with a `convenient' choice of two generators.

We conclude this section with the analogues of Corollaries \ref{cor:Z2psi} and \ref{cor:V(H)} for $\Spin(2n)$. The proofs are almost identical and left for the reader. 

\begin{corollary}
The $\Z/2$-graded $\Psi$-ring $\K^*(\Spin(2n)) \o\Q$ is generated by two elements.
\end{corollary}

\begin{corollary} \label{cor:V(H)spin}
Let $H$ be a Lie group and suppose $f, g : H \to \Spin(2n)$ are homomorphisms of Lie groups.  Let $u, v$ be representations whose images generate $V(\Spin(2n))\o \Q$ as a $\Psi$-module $V(\Spin(2n))\o\Q$. Then the two homomorphisms
$$
f_*, g_* : \pi_*(H)\o \Q \to \pi_*(\Spin(2n))\o \Q
$$
are equal if and only if 
$$
f^*(u) = g^*(u) \text{ and } f^*(v) = g^*(v) \in V(H) \o \Q. 
$$
\end{corollary}

\section{On the rational cohomology of $BG$ and $G/K$} \label{sec:ratcoho}
We now use the results of the previous section to throw some light on the rational cohomology of $BG$ and the homogeneous spaces $G/K$ where $G$ and $K$ have the same rank. Throughout this section $G$ is a simply connected simple Lie group of type
$$
(2r_1-1, \ \ldots, \ 2r_n -1), 
$$
where $n$ is the rank of $G$. We start by assuming the integers $r_1, \ldots, r_n$ are distinct, thus excluding the groups $\Spin(2n)$. These are dealt with in Section \ref{sec:spin2nagain}. 

\subsection{On the rational cohomology of $BG$.}
 It is well known, see \cite[Th\'eor\`eme~19.1]{MR0051508}, that \[ H^*(BG; \Q) = \Q [a_{2j} : j = r_1, \dots, r_n], \]
where $a_{2j} \in H^{2j}(BG; \Q)$.  Of course, these generators are not unique but we only need one generator in each of the dimensions $2j$ for $j = r_1, \dots, r_n$.

From now on, if $\rho$ is a representation of $G$ we will write $\ch_k(\rho) \in H^{2k}(BG; \Q)$ for the $k$-th component of the Chern character of the vector bundle over $BG$ associated to $\rho$.  We extend this convention to any characteristic class and indeed to the associated bundle over $X/G$, where $G$ acts freely on a compact Hausdorff space $X$.

We now ask what might constitute a good choice of generators. Our basic strategy is to use representation theory to do computations. So we are led to the following question.  Is there a representation $u$ of $G$ such that the elements $\ch_k(u)$ where $k \geq 1$ generate the ring $H^*(BG; \Q)$?  

\begin{theorem} \label{thm:chernofminisenough}
If $u \in R(G)$, then $H^*(BG; \Q)$ is generated as a $\Q$-algebra by
$$
\ch_{r_1}(u), \ \dots , \ \ch_{r_n}(u)
$$
if and only if the element of $V(G)\o\Q$ defined by $u$ generates $V(G)\o\Q$ as a $\Psi$-module.
\end{theorem}

\begin{proof}
If $X$ is a CW complex of finite type and $x \in \K^0(X)$ then 
$$
\ch_q(\psi^kx) = k^q\ch_q(x).
$$
So we make $H^{\ev}(X ; \Q)$ into a $\Psi$-ring by defining Adams operations as follows.  For $y \in H^{2q}(X; \Q)$ define $\psi^k(y)$ by
$$
\psi^k(y) = k^qy.
$$
It follows that 
$$
\ch : \K^0(X) \o \Q \to H^{\ev}(X; \Q)
$$
is an isomorphism of $\Psi$-rings.

We use the ad hoc notation $W(G)$ for the indecomposable quotient of the $\Psi$-ring $H^{\ev}(BG; \Q)$.  It is isomorphic to the $\Psi$-module $\Q(r_1, \dots, r_n)$.  We get a basis for this $\Psi$-module by choosing homogeneous generators $a_{r_1},  \dots, a_{r_n}$ for the graded ring $H^*(BG; \Q)$.  We retain the notation 
$$
\ch : V(G) \o\Q \to W(G)
$$
for the homomorphism of $\Psi$-modules defined by the Chern character.  This homomorphism is an isomorphism of $\Psi$-modules, since $\ch : \K^0(BG) \o \Q \to H^{\ev}(BG; \Q)$ is an isomorphism of $\Psi$-rings. Lemma \ref{lem:cyclicpsimodules} completes the proof.
\end{proof}



\subsection{On the rational cohomology of homogeneous spaces}
Let $K$ be a subgroup of $G$ and $i:K \to G$ be the inclusion map. We let $\pi = Bi : BK \to BG$ be the induced map of classifying spaces defined by $i$. There is the usual fibration
\[
\begin{CD}
G/K @>j>> BK @>{\pi}>> BG
\end{CD}
\]

\begin{theorem} \label{thm:BorelandPittie}
Suppose the rank of $K$ is equal to the rank of $G$ and $u \in R(G)$ is a representation whose image generates the $\Psi$-module $V(G)\o\Q$.  Then the rational cohomology of $G/K$ is given by 
$$
H^*(G/K; \Q) = H^*(BK; \Q)/I
$$
where $I$ is the ideal in $H^*(BK; \Q)$ generated by the elements
$$
i^*(\ch_q(u)) = \ch_q(i^*(u)), \mathrm{\ with \ } q \geq 1.
$$
\end{theorem}

\begin{proof}
Consider the following commutative diagram.
$$
\begin{CD}
R(K) \otimes\Q @>>> \K^*(G/K) \otimes \Q @>{\ch}>> H^*(G/K; \Q) \\
@AAA @A{j^*}AA @AAj^*A \\
R(K) \otimes\Q @>>> \K^*(BK; \Q) @>>{\ch}> H^*(BK; \Q)
\end{CD}
$$
By Theorem \ref{thm:Pittiesthm}, $R(K) \otimes \Q \to \K^*(G/K) \otimes \Q$ is surjective. Since the Chern character is a rational isomorphism it follows that the composite of the two homomorphisms in the top row is surjective. The commutativity of the diagram shows that $j^*$ is surjective. A standard application of the Leray--Hirsch theorem as in \cite[Section~26]{MR0051508} shows that the kernel of $j^*$ is the ideal in $H^*(BK; \Q)$ generated by the elements
$$
\pi^*(x), \text{ with } x \in H^q(BG) \mathrm{\ and \ }  q \geq 1.
$$
Applying Theorem \ref{thm:chernofminisenough} completes the proof. 
\end{proof}

\subsection{The groups $\Spin(2n)$} \label{sec:spin2nagain}

Now we give the analogous results for $G = \Spin(2n)$. The proofs are almost identical and left to the reader.  

\begin{theorem} \label{thm:chernofminisenoughSpin}
Let $u, v \in R(G)$ be representations whose images generate the $\Psi$-module $V(G)\o\Q$. The cohomology algebra $H^*(BG; \Q)$ is generated as a $\Q$-algebra by
\[ \ch_{r_1}(u), \ldots, \ch_{r_n}(u),  \ch_{r_1}(v), \ldots, \ch_{r_n}(v).\]
\end{theorem}

\begin{theorem} \label{thm:BorelandPittieSpin}
Suppose the rank of $K$ is also $n$ and that $u,v \in R(G)$ are representations whose images generate the $\Psi$-module $V(G)\o\Q$. Then the rational cohomology of $G/K$ is given by 
\[
H^*(G/K; \Q) = H^*(BK; \Q)/I
\]
where $I$ is the ideal in $H^*(BK; \Q)$ generated by the elements
\[
\ch_q(i^*(u)), \mathrm{\ and \ } \ch_q(i^*(v)), \mathrm{\ with \ } q \geq 1. \]
\end{theorem}

\subsection{An example}
The Grassmann manifold $G_k(\C^n)$ is the homogeneous space
$$
\U(n)/(\U(k) \times \U(n-k)).
$$
We assume that $k \leq n-k$. 

Write $V_m$ for the universal vector bundle over $B\U(m)$.  This is the bundle defined by the usual action of $\U(m)$ on $\C^m$. There are two obvious bundles $V_k$ and $V_{n-k}$ on $B(\U(k) \times \U(n-k))$ and $H^*(B(\U(k) \times \U(n-k)); \Q)$ is the polynomial algebra generated by
$$
a_2 = \ch_1(V_k), \dots, a_{2k} = \ch_k(V_k),
$$$$
b_2 = \ch_1(V_{n-k}), \dots, b_{2(n-k)} = \ch_{n-k}(V_{n-k}).
$$
The restriction of $V_n$ to $B(\U(k) \times \U(n-k))$ is $V_k \oplus V_{n-k}$.  So the relations given by Theorem \ref{thm:BorelandPittie}
are simply
$$
\ch_i(V_k) + \ch_i(V_{n-k}) = 0 \quad \text{for $1 \leq i \leq n$.}
$$
This gives a presentation of $H^*(G_k(\C^n); \Q)$ with $n$ generators and $n$ relations.

We should decode these relations.   The first $k$ relations simply say that
$$
a_{2i} + b_{2i} = 0 \quad \text{for $1 \leq i \leq k$.}
$$
Now recall that $\ch_{k+1}(V_k)$ is not necessarily zero.  But it is a polynomial in $\ch_1(V_k), \dots , \ch_{k}(V_k)$. So $\ch_{k+1}(V_k) + \ch_{k+1}(V_{n-k}) = 0$ tells us that $b_{2k+2}$ is a polynomial in $a_2, \dots , a_{2k}$.  This carries on to show that $b_{2k+2}, \dots , b_{2(n-k)}$ are also polynomials in $a_2, \dots, a_{2k}$.  These polynomials are given by the Newton identities for symmetric polynomials.  For $i = n-k+1, \dots, n$ the equations
$$
\ch_i(V_k) + \ch_i(V_{n-k}) = 0
$$
tell us that a particular polynomial in $a_2, \dots, a_{2k}$ is equal to another particular polynomial in $b_2, \dots, b_{2(n-k)}$.  Now when we substitute the previous formulas for the $b_{2i}$ in terms of the $a_{2j}$ this shows that two polynomials in the $a_{2i}$ are equal.  This gives $k$ relations in the $a_{2i}$. We end up with a presentation with $k$ generators and $k$ relations.  

This calculation does not use the Schubert cell decomposition of the Grassmann manifold.

\section{Rationally elliptic spaces.} \label{sec:ratelliptic}
A simply connected topological space is {\em rationally elliptic} if both $H^*(X;\Q)$ and $\pi_*(X) \o \Q$ are finite dimensional. Any compact simply connected homogeneous space is rationally elliptic.  The space $S^3 \vee S^3$ is a simple example of a space that is not rationally elliptic: its rational cohomology is finite dimensional but its rational homotopy is not.  The space $\CP^{\infty}$ is not rationally elliptic: its rational homotopy is finite dimensional but its rational cohomology is not.   If $X$ is rationally elliptic then we write
$$
\chi_H(X), \quad \chi_{\pi}(X)
$$
for the Euler characteristic of $H^*(X;\Q)$ and $\pi_*(X) \o \Q$, respectively.

A very thorough account of the theory of rationally elliptic spaces can be found in \cite[Section~32]{felix2001rational} together with references to the original papers. The following theorem is a combination of Propositions 32.10 and 32.16 in \cite{felix2001rational}. 

\begin{theorem} \label{thm:ellipticprops} Suppose that $X$ is a rationally elliptic space. Then
$$
\chi_H(X) \geq 0, \quad \chi_{\pi}(X) \leq 0.
$$
Furthermore, the following statements are equivalent:
\begin{enumerate}
\item
$\chi_H(X) > 0$,
\item
$H^*(X; \Q)$ is concentrated in even degrees,
\item
$H^*(X; \Q)$ is the quotient of a polynomial ring $\Q[a_1, \dots, a_q]$ where the $a_i$ have even degree, by an ideal generated by a regular sequence of length $q$,
\item
$ \dim \pi_{\text{even}}(X) \o \Q = \dim \pi_{\text{odd}}(X) \o \Q$, which is equivalent to $\chi_\pi(X) = 0$.
\end{enumerate}
If these conditions hold then the Poincar\'{e} polynomial of $X$ is
\[P_X(t) = \frac{\prod\limits_{i = 1}^q (1-t^{2m_i})}{\prod\limits_{i = 1}^q (1-t^{2n_i})},\] and the Euler characteristic of $X$ is \[ \chi_H(X) = \frac{\prod\limits_{i = 1}^q m_i }{\prod\limits_{i = 1}^q n_i}, \]
where $2n_1, \ldots, 2n_q$ and $2m_1-1, \ldots, 2m_q-1$ are the degrees of a basis of $\pi_*(X) \otimes \Q$. 

\end{theorem}
Therefore, the minimum number of generators, the integer $q$ in (iii), is equal to $\dim \pi_{\text{even}}(X) \o \Q$.

\section{The rational homotopy groups of the Rosenfeld projective planes} \label{sec:rathomRosenfeld}

The rational homotopy groups of the Rosenfeld projective planes were first calculated by Svjetlana Terzi\'{c} in \cite{MR2906538}. We give a proof following the main strategy of this paper, using Theorem \ref{thm:psimodulehomsforinclusions} and a computation in representation theory.

\begin{theorem} \label{thm:rathomgrpsRosenfeld}
The non-zero rational homotopy groups of $R_5$, $R_6$ and $R_7$ are given by
\begin{enumerate}
\item $\pi_m(R_5)\otimes\Q = \Q$, if $m = 2, 8, 17, 23$,
\item $\pi_m(R_6)\otimes\Q = \Q$, if $m = 4, 8, 12, 23, 27, 35$,
\item $\pi_m(R_7)\otimes\Q = \Q$, if $m = 8,12, 16, 20, 35, 39, 47, 59$.
\end{enumerate}
\end{theorem}

Recall that in \cite[Theorem~6.1]{MR1428422}, Adams shows how to construct injective homomorphisms 
\[
h_6 : \Spin(10) \to E_6, \quad h_7 : \Spin(12) \to E_7,
\] 
and a homomorphism
\[
\quad h_8 : \Spin(16) \to E_8,
\]
with kernel a central subgroup of order $2$ that is not contained in the kernel of the universal covering $\Spin(16) \to \SO(16)$.

\begin{lemma} \label{lem:kernelsofhmaps}
\begin{enumerate}
\item The kernel of $(h_6)_* : V^*(\Spin(10)) \otimes \Q \to V^*(E_6) \otimes \Q$ has dimension~$1$.
\item The kernel of $(h_7)_* : V^*(\Spin(12)) \otimes \Q \to V^*(E_7) \otimes \Q$ has dimension $2$.
\item The kernel of $(h_8)_* : V^*(\Spin(16)) \otimes \Q  \to V^*(E_8) \otimes \Q$ has dimension $4$.
\end{enumerate}
\end{lemma}
The proof is computational and we postpone it until Section 9.

\subsection{The rational homotopy groups of $R_5$}
Recall that $R_5 = E_6/(\Spin(10)\times_{C_4}S^1)$.  Let $N_5$ be the homogeneous space $E_6/\Spin(10)$, so $N_5$ is an $S^1$-bundle over $R_5$.  

First we use the four-term exact sequences of Section 3.2 to compute the rational homotopy groups of $N_5$. The type of $\Spin(10)$ is
\[
(3, 7, 9, 11, 15)
\]
and the type of $E_6$ is
\[
(3, 9, 11, 15, 17, 23).
\]
The following table summarises the four term exact sequences for $N_5$. 
\vspace{\baselineskip}
\begin{center}
\begin{tabular}{ | c | c | c | c | c | c | c | c | }
\hline
$r$  & $3$ & $7$ & $9$ & $11$ & $15$ & $17$ & $23$  \\
\hline
$\pi_{r+1}(N_5)$ & & $\Q$ & &  & & &     \\
\hline
$\pi_{r}(\Spin(10))$ & $\Q$ & $\Q$ & $\Q$ & $\Q$ & $\Q$ & &   \\
\hline
$\pi_{r}(E_6)$ & $\Q$ & &$\Q$ &$\Q$ & $\Q$ &$\Q$ & $\Q$ \\
\hline
$\pi_{r}(N_5)$ &  &  &  &   & & $\Q$ & $\Q$  \\
\hline
\end{tabular}
\end{center}
\vspace{\baselineskip}
Our conventions are
\begin{enumerate}
\item homotopy groups are rational homotopy groups,
\item a blank entry in the table means the relevant group is $0$, as does not appearing in the table.
\end{enumerate}
The two middle rows are the homotopy groups of $\Spin(10)$ and $E_6$, so we must calculate the homomorphism
$$
\pi_r(\Spin(10)) \to \pi_r(E_6)
$$
for $r = 3,7,9,11,15$.  Using Theorem \ref{thm:psimodulehomsforinclusions} and Lemma \ref{lem:kernelsofhmaps} we know that the kernel of $\pi(\Spin(10)) \to \pi(E_6)$ is one dimensional, which must occur when $r=7$ since $\pi_7(E_6) = 0$. Thus, for $r = 9, 11, 15$ this homomorphism is injective and hence an isomorphism. This gives the complete table.

A simple argument with the homotopy exact sequence of the circle bundle $N_5 \to R_5$ completes the calculation of the rational homotopy groups of $R_5$.

\subsection{The rational homotopy groups of $R_6$}
Recall that $R_6 = E_7 /(\Spin(12) \times_{C_2} S^3)$.  Let $N_6$ be the homogeneous space $E_7/\Spin(12)$, so $N_6$ is an $S^3$-bundle over $R_6$. 

As before we begin by computing the homotopy groups of $N_6$ using the four-term exact sequences.  The type of $\Spin(12)$ is
\[
(3, 7, 11, 11, 15, 19)
\]
and the type of $E_7$ is 
\[
(3, 11, 15, 19, 23, 27, 35).
\]
The following table, with the same conventions as the previous case, summarises the four term exact sequences of $N_6$.
\vspace{\baselineskip}
\begin{center}
\begin{tabular}{ | c | c | c | c | c | c | c | c | c | }
\hline
$r$  & $3$ & $7$ & $11$ & $15$ & $19$ & $23$ & $27$ & $35$  \\
\hline
$\pi_{r+1}(N_6)$ & & $\Q$ & $\Q$ &  & & &  &     \\
\hline
$\pi_{r}(\Spin(12))$ & $\Q$ & $\Q$ & $\Q\oplus\Q$ & $\Q$ & $\Q$ & & &    \\
\hline
$\pi_{r}(E_7)$ & $\Q$ & &$\Q$ &$\Q$ &$\Q$ & $\Q$ &$\Q$ & $\Q$  \\
\hline
$\pi_{r}(N_6)$ &  &  &  &   & & $\Q$ & $\Q$ &$\Q$  \\
\hline
\end{tabular}
\end{center}

\vspace{\baselineskip}
This time we must calculate the homomorphism
$$
\pi_r(\Spin(12)) \to \pi_r(E_7).
$$
Theorem \ref{thm:psimodulehomsforinclusions} and Lemma \ref{lem:kernelsofhmaps} show that the kernel of $\pi(\Spin(12)) \to \pi(E_7)$ is two dimensional. Again, it is clear that it must have one dimensional kernel in degrees $r = 7, 11$. The table then follows. The homotopy exact sequence of the $S^3$ bundle $N_6 \to R_6$ completes the proof.

\subsection{The rational homotopy groups of $R_7$}
First we must explain what the half-spin or semi-spin group $\HSpin(2n)$ is.  The centre of $\Spin(4n)$ is $C_2 \times C_2$. So there are three non-trivial central subgroups of order two in $\Spin(4n)$.  If $4n \neq 8$ the outer automorphism group of $\Spin(4n)$ is cyclic of order two. The non-trivial outer automorphism interchanges two of these central subgroups, and this gives two different, but isomorphic, $C_2$-quotients of $\Spin(4n)$.  These are the half-spin groups.  It is usual to choose one to work with and call this $\HSpin(4n)$.  The third quotient is $\SO(4n)$.  In the case of $\Spin(8)$, the outer automorphism group is the symmetric group $\Sigma_3$ and this acts transitively on the three central subgroups of order $2$.  So all three $C_2$ quotients of $\Spin(8)$ are isomorphic to $\SO(8)$.

Now $R_7 = E_8/\HSpin(16)$, the type of $\HSpin(16)$ is
 $$
(3, 7, 11, 15, 15, 19, 23, 27),
$$
and the type of $E_8$ is
$$
(3, 15, 23, 27, 35, 39, 47, 59)
$$
The following table, with the usual conventions,  summarises the conclusions of the calculations with the four-term exact sequences for $R_7$. 
\vspace{\baselineskip}
\begin{center}
\begin{tabular}{ |c | c | c | c | c | c | c | c | c | c | c | c | }
\hline
$r$ & $3$ & $7$ & $11$ & $15$ & $19$ & $23$ & $27$ & $35$ & $39$ & $47$ & $59$ \\
\hline
$\pi_{r+1}(R_7)$ &  & $\Q$ & $\Q$ & $\Q$ &$\Q$  &  &  &  &  &  & \\
\hline
$\pi_{r}(\HSpin(16))$ &$\Q$ & $\Q$ & $\Q$ & $\Q\oplus \Q$ & $\Q$ & $\Q$ & $\Q$ & & & & \\
\hline
$\pi_{r}(E_8)$ & $\Q$ & & & $\Q$ & & $\Q$ & $\Q$ & $\Q$ & $\Q$ & $\Q$ & $\Q$ \\
\hline
$\pi_{r}(R_7)$ &  &  &  &  &  &  & & $\Q$ & $\Q$ & $\Q$ & $\Q$ \\
\hline
\end{tabular}
\end{center}
\vspace{\baselineskip}
This time we must compute 
$$
\pi_r(\HSpin(16)) \to \pi_r(E_8),
$$ 
which can be done with the same arguments as in the previous two cases.  

\subsection{What does the theory of elliptic spaces tell us?}
Applying the results in Section 6, we read off the following information about the rational cohomology of the Rosenfeld projective planes $R_5, R_6, R_7$.
\begin{enumerate}
\item\[ 
H^*(R_5, \Q) = \frac{\Q[a_2, a_8]}{(r_{18},r_{24})} \]
\[
P = \left(\frac{1-t^{18}}{1-t^2}\right)\left(\frac{1-t^{24}}{1-t^8}\right), \quad \chi = 27.
\]
\item
\[ H^*(R_6, \Q) = \frac{\Q[a_4, a_8, a_{12}]}{(r_{24}, r_{28}, r_{36})}\]
\[ P = \left(\frac{1-t^{24}}{1-t^{8}} \right)\left(\frac{1-t^{28}}{1-t^4}\right)\left(\frac{1-t^{36}}{1-t^{12}}\right), \quad \chi = 63. \]
\item \[ H^*(R_7, \Q) = \frac{\Q[a_8, a_{12}, a_{16}, a_{20}]}{(r_{36}, r_{40}, r_{48}, r_{60})} \] 
\[ P = \left(\frac{1-t^{36}}{1-t^{12}}\right)\left(\frac{1-t^{40}}{1-t^{8}}\right)\left(\frac{1-t^{48}}{1-t^{16}}\right)\left(\frac{1-t^{60}}{1-t^{20}}\right), \quad \chi = 135.\] 
\end{enumerate}

In all three cases the subscripts on the generators and the relations are the degrees.   Furthermore, in all three cases the relations  form a regular sequence.  

These results agree with \cite[Table~A]{MR3751970} which also contains references to other proofs.  It is worth repeating here that one of our main aims is to give a systematic account of these results.  In the next section we show how to get explicit presentations for the rings $H^*(R_5; \Q)$, $H^*(R_6; \Q)$ and $H^*(R_7; \Q)$ by using Theorem \ref{thm:BorelandPittie}.  

\section{Explicit presentations of the rational cohomology algebras of $R_5, R_6, R_7$} \label{sec:ratcohR567}

\subsection{The general strategy} \label{subsec:genstratcomps}

Let $G$ be a Lie group and let $K$ be a subgroup of $G$. Let $i : K \to G$ be the inclusion of $K$ in $G$. There is a fibration
$$
\begin{CD}
G/K @>j>> BK   @>{\pi = Bi}>>   BG
\end{CD} 
$$
Borel \cite[Section~26]{MR0051508} proves that if $K$ has the same rank as $G$, then:
\begin{enumerate} 
\item
$j^* : H^*(BK; \Q) \to H^*(G/K; \Q)$ is surjective,  
\item
$\pi^* : H^*(BG; \Q) \to H^*(BK; \Q)$ is injective,
\item
$H^*(G/K; \Q) = H^*(BG; \Q)/I$
where $I$ is the ideal in $H^*(BK; \Q)$ generated by
$$
\pi^*(x) \text{ with } x \in H^q(BK; \Q) \text{ and }  q \geq 1. 
$$
\end{enumerate}
The classical way of calculating $\pi^*$ is to use Weyl invariants. Choose a maximal torus 
\[
T  \subset K \subset G
\] 
and let $W_K \subset W_G$ be the Weyl groups of $K, G$. Then $W_K$ acts on $H^*(BT; \Q)$ and the invariants $H^*(BT; \Q)^{W_K}$ can be identified with $H^*(BK; \Q)$. Similarly, the invariants $H^*(BT; \Q)^{W_G}$ can be identified with $H^*(BG; \Q)$. 
\begin{enumerate}
\item The elements of $H^*(BT; \Q)^{W_K}$ give generators of $H^*(G/K; \Q)$.
\item The elements of $H^*(BT; \Q)^{W_G}$ give relations between these generators. 
\item Given $x \in H^*(BT; \Q)^{W_G}$, it corresponds to the relation among the generators given by expressing $x$ as a polynomial in the elements of $H^*(BT; \Q)^{W_K}$.
\end{enumerate}

This is a very effective way of doing the computations when $G$ and $K$ are classical groups, as the famous papers of Borel and Hirzebruch show \cite{MR0102800,MR0110105,MR0120664}. However, it gets more complicated for the exceptional Lie groups. For example, the order of $W_{E_8}$ is $696,729, 600$, just short of $700$ million, and calculating in 
$$
\Q[x_1, \ldots, x_8]^{W_{E_8}}
$$
is not likely to be straightforward. 

Our strategy is to use representation theory, $K$-theory, and Adams operations to give an alternative approach to calculating $H^*(G/K; \Q)$. Let $G$ be a simple, simply connected Lie group with the exception of $\Spin(2n)$. 

Theorem \ref{thm:chernofminisenough} gives a necessary and sufficient condition for a representation $u$ of $G$ to have the property that the cohomology ring $H^*(BG; \Q)$ is generated by the elements  
$$
\ch_q(u) \in H^{2q}(BG; \Q),  \text{ with } q \geq 1.
$$
Theorem \ref{thm:simplesinglegenpsimod} shows there exists such a representation and we give an explicit choice of one in Theorem \ref{thm:minmodgen}. Now, Theorem \ref{thm:BorelandPittie} then shows that the kernel of $j^*$ is the ideal of $H^*(BK; \Q)$ generated by
\[
\ch_q(i^*(u)) = i^*(\ch_q(u)),  \text{ with } q \geq 1. 
\]
Therefore, we must compute $\ch_q(i^*(u))$. When $K$ is one of the classical groups we are now in the realm of standard calculations with characteristic classes in algebraic topology, see for example \cite{MR1443417}. 

In the case where $G = \Spin(2n)$ we modify this method in the obvious way, using the results in Section \ref{sec:spin2l}. 

In each of our three examples there are Lie groups $N, G$ of the same rank and a homomorphism $h : N \to G$ with finite kernel, $C$.  The homogeneous space we want to study is $G/K$ where $K$ is the image of $h$. We use the notation $p : N \to N/C =K$ for the quotient homomorphism and $i: K \to G$ for the inclusion of $K$ in $G$. 

Now we form the commutative diagram of spaces

\[
\begin{CD}
@.BN @>Bh>> BG \\
@.@VBpVV @VV=V\\
G/K @>>j>BK @>>Bi> BG\\
\end{CD}
\]
where $j$ is the inclusion of the fibre of the fibration $BK \to BG$. Since the kernel of $p$ is finite 
$$
(Bp)^*: H^*(BK; \Q) \to H^*(BN; \Q)
$$
is an isomorphism.  This gives us a surjective ring homomorphism
$$
\phi = j^*\circ (Bp^*)^{-1}: H^*(BN; \Q) \to H^*(G/K; \Q).
$$
It follows from Borel's theorem that the kernel of $\phi$ is the ideal generated by
\[
Bh^*(x) \text{ with } x \in H^q(BG; \Q) \text{ and } q \geq 1.
\]

\subsection{The rational cohomology of $R_5$}

In \cite[Chapter~8, p.51]{MR1428422}, Adams shows that there is a homomorphism $h : \Spin(10) \times \U(1) \to E_6$ with kernel a central subgroup $C_4$. By definition
\[
R_5 = E_6/(\Spin(10) \times_{C_4} \U(1)),
\]
so, as in Section \ref{subsec:genstratcomps},  we get a surjective ring homomorphism
\[
\phi : H^*(\Spin(10) \times \U(1); \Q) \to H^*(R_5; \Q).
\]

Next, using \cite[Corollary~8.3]{MR1428422},  we define $u$ to be the 27-dimensional fundamental representation of $E_6$ such that
$$
h^*(u) = \xi^4 + v_{10} \otimes \xi^2 + \Delta_{10}^+ \otimes \xi^{-1}.
$$
Here $v_{10}$ is the $10$-dimensional vector representation of $\Spin(10)$, $\Delta_{10}^{\pm}$ are the two spin representations of $\Spin(10)$ and $\xi$ is the vector representation of $\U(1)$. 

It follows from Theorems \ref{thm:minmodgen} and \ref{thm:BorelandPittie} that $\phi$ defines an isomorphism
$$
H^*(B\Spin(10)\times U(1); \Q)/ I  \to H^*(R_5; \Q)
$$
where $I$ is the ideal in $H^*(B\Spin(10)\times U(1); \Q)$ generated by the elements
$$
 \ch_q(h^*(u)), \text{with $q \geq 1$}.
$$

Finally we compute 
$$
\ch_q(h^*(u)) \in H^{2q}(B(\Spin(10) \times \U(1)); \Q)
$$ 
in terms of the Pontryagin classes $p_i(v_{10}) \in H^{4i}(B\Spin(10); \Q)$, the Euler class  $e(v_{10}) \in H^{10}(B\Spin(10); \Q)$, and the first Chern class $c_1(\xi)$ in $H^2(B\U(1); \Q)$. This leads to the following theorem.

\begin{theorem} \label{thm:presR5}
The cohomology ring $H^*(R_5; \Q)$ is isomorphic to \[ \frac{\Q[a_{2}, a_{8}]}{(r_{18},r_{24})}, \] 
where the generators are 
\[
a_2 = \phi(c_1(\xi)), \quad a_8 = \phi(p_2(v_{10})),
\]  and the relations are 
\begin{align*}
r_{18} & = -39936 a_2^9 + 1728 a_2^5 a_8 + a_2 a_8^2, \\
r_{24} & = -50429952 a_2^{12} + 3068928 a_2^8 a_8 - 11808 a_2^4 a_8^2 - a_8^3.
\end{align*}
\end{theorem}

The integral cohomology ring of $R_5$ has been calculated by Toda and Watanabe, given in \cite[Corollary~C]{MR0358847}. The corresponding presentation of the rational cohomology algebra they obtain is 
\[ \frac{\Q[t,w]}{(s_{18},s_{24})}, \]
where $s_{18} = t^9 - 3w^2t$ and $s_{24} = w^3 - 9wt^8 +15w^2t^4$. Let $\psi:\Q[t,w] \mapsto \Q[a_2, a_8]$ be the (graded) algebra homomorphism determined by $\psi(t) = 4a_2$ and $\psi(w) = a_8/6 + 144a_2^4$. One can use \Magma{}, for example, to check that the ideal of $\Q[a_2,a_8]$ generated by $\psi(s_{18})$, $\psi(s_{24})$ is equal to the ideal generated by $r_{18}$, $r_{24}$. Therefore, the two algebras are isomorphic.

\subsection{The rational cohomology of $R_6$}

Adams shows in \cite[Chapter~8, p.50]{MR1428422} that there is a homomorphism
$$
h : \Spin(12) \times \Sp(1) \to E_7
$$
such that the kernel of $h$ is a central subgroup of order $2$. By definition 
$$
R_6 = E_7/(\Spin(12) \times_{C_2} \Sp(1)),
$$
and so we get a surjective ring homomorphism
$$
\phi  : H^*(B\Spin(12) \times B\Sp(1); \Q) \to H^*(R_6; \Q).
$$
Now let $u$ be the 56-dimensional fundamental representation of $E_7$. By [Corollary 8.2]\cite{MR1428422}, 
\[
h^*(u) = v_{12} \o \zeta + \Delta_{12}^+.
\]
Here $v_{12}$ is the $12$ dimensional vector representation of $\Spin(12)$, $\Delta^{\pm}_{12}$ are the spin representations, and $\zeta$ is the representation of $\Sp(1)$ on $\H = \C^2$.  

Again, we use Theorems \ref{thm:minmodgen} and \ref{thm:BorelandPittie}, to show that $\phi$ defines an isomorphism
$$
H^*(B\Spin(12)\times \Sp(1); \Q)/ I  \to H^*(R_6; \Q),
$$
where $I$ is the ideal in $H^*(B\Spin(12)\times \Sp(1); \Q)$ generated by the elements
\[
 \ch_q(h^*(u)), \text{with $q \geq 1$.}
\]
So we must compute
$$
\ch_q(h^*(u)) \in H^{2q}(B(\Spin(12) \times \Sp(1); \Q)
$$
in terms of the Pontryagin classes $p_i(v_{12}) \in H^{4i}(B\Spin(12); \Q)$, the Euler class $e(v_{12}) \in H^{12}(B\Spin(12); \Q)$, and the Chern class $c_2(\zeta) \in H^4(B\Sp(1); \Q)$.

\begin{theorem} \label{thm:presR6}
The cohomology ring $H^*(R_6; \Q)$ is isomorphic to \[ \frac{\Q[a_{4}, a_{8}, a_{12}]}{(r_{24},r_{28}, r_{36})}, \] 
where the generators are 
\[
a_4 = \phi(p_1(v_{12})) = \phi(2c_2(\zeta)), \quad a_8 = \phi(p_2(v_{12})), \quad a_{12} = \phi(p_3(v_{12})),
\]  and the relations are 
\begin{align*}
\ r_{24} & = 38367 a_4^6 - 131436 a_4^4 a_8 + 88272a _4^2 a_8^2 - 1600 a_8^3 -
273024 a_4^3 a_{12} + 55296 a_4 a_8 a_{12} - 10368 a_{12}^2, \\
r_{28} & = 63 a_4^7 + 96 a_4^5 a_8 + 48 a_4^3 a_8^2 + 640 a_4 a_8^3 - 1686 a_4^4 a_{12} - 4656 a_4^2 a_8 a_{12} + 160 a_8^2 a_{12} -
 1152 a_4 a_{12}^2, \\
r_{36} & = 19503 a_4^9 + 41184 a_4^7 a_8 - 150816 a_4^5 a_8^2 - 156672 a_4^3 a_8^3 + 19200 a_4 a_8^4 - 127224 a_4^6 a_{12} - {} \\ 
 & \phantom{==} 908448 a_4^4 a_8 a_{12} + 264576 a_4^2 a_8^2 a_{12} + 12800 a_8^3 a_{12} - 1806336 a_4^3 a_8^2 + 36864 a_4 a_8 a_{12}^2 + 18432 a_{12}^3.
\end{align*}
\end{theorem}

A presentation of $H^*(R_6; \Q)$ is given in \cite[Lemma~2.4(ii)]{MR1878720} by Nakagawa. Ours is different but, as in the case of $R_5$, it can be checked that the two presentations give isomorphic algebras over $\Q$. An explicit map is given by $\psi: \Q[a,b,c] \rightarrow \Q[a_4,a_8,a_{12}]$ determined by $\psi(a) = 120a_4$, $\psi(b) =  -11700a_4^2 + 1200a_8$, $\psi(c) = -993600a_4^3 + 14400a_4a_8 + 21600a_{12}$.

\subsection{The rational cohomology of $R_7$}

In \cite[Chapter 7]{MR1428422} Adams shows that there is a homomorphism 
$$
h : \Spin(16) \to E_8
$$
with kernel a cyclic group of order 2. The quotient of $\Spin(16)$ by this subgroup is the half-spin group $\HSpin(16)$ and by definition
\[
R_7 = E_8/\HSpin(16).
\]
So we get a surjective ring homomorphism
$$
\phi : H^*(B\Spin(16); \Q) \to H^*(R_7; \Q).
$$

Now let $u$ be the adjoint representation of $E_8$.  Then from \cite[Chapter 7]{MR1428422} it follows that
$$
h^*(u) = \lambda^2(v_{16}) \oplus \Delta^+.
$$
Then, as in the two previous two sections, it follows that we get an isomorphism
$$
\phi : H^*(B\Spin(16); \Q)/I \to H^*(R_7; \Q) 
$$
where $I$ is the ideal generated by
$$
\ch_q(h^*(u)) \in H^{2q}(B(\Spin(16); \Q), \text{with $q \geq 1$.}. 
$$
Finally, we have to compute $\ch_q(h^*(u))$ in terms of the Pontryagin classes $p_i(v_{16})$ and the Euler class $e(v_{16})$.

\begin{theorem} \label{thm:presR7}
The cohomology ring $H^*(R_7; \Q)$ is isomorphic to \[ \frac{\Q[a_{8}, a_{12}, a_{16}, a_{20}]}{(r_{36},r_{40}, r_{48},r_{60})}, \] 
where the generators are \[a_8 = \phi(p_2(v_{16})), \quad a_{12} = \phi(p_3(v_{16})), \quad a_{16} = \phi(p_4(v_{16})) \quad a_{20} = \phi(p_5(v_{16})),\] and the relations are 
\begin{align*}
\ r_{36} & = 275 a_8^3 a_{12} - 6150 a_8^2a_{20} + 5400 a_8 a_{12} a_{16} - 756 a_{12}^3 - 10800 a_{16}a_{20}, \\
r_{40} & = 275 a_{8}^5 + 4080 a_{8}^3 a_{16} + 945 a_{8}^2 a_{12}^2 - 26460 a_8 a_{12} a_{20} - 25920 a_8 a_{16}^2 + 27216 a_{12}^2 a_{16} + 26460 a_{20}^2, \\
r_{48} & = -225875 a_8^6 - 8037000 a_{8}^4 a_{16} + 4233600 a_8^3 a_{12}^2 - 23020200 a_{8}^2 a_{12} a_{20} - 29160000 a_{8}^2 a_{16}^2 + {} \\ 
 & \phantom{==} 28576800 a_{8} a_{12}^2a_{16} - 166698000 a_{8} a_{20}^2 + 3000564 a_{12}^4 + 57153600 a_{12} a_{16} a_{20} + 46656000 a_{16}^3, \\
r_{60}  & = -2868125 a_8^6a_{12} + 22312500 a_{8}^5 a_{20} - 36945000 a_{8}^4 a_{12} a_{16} - 3307500 a_{8}^3 a_{12}^3 - {} \\
& \phantom{==} 390600000 a_8^3 a_{16} a_{20} + 222264000 a_8^2 a_{12}^2 a_{20} - 243000000 a_{8}^2 a_{12} a_{16}^2 + 71442000 a_{8} a_{12}^3 a_{16} - {} \\
& \phantom{==} 972405000 a_{8} a_{12} a_{20}^2 + 1360800000 a_{8}a_{16}^2 a_{20} - 18003384 a_{12}^5 + 1000188000 a_{12}^2 a_{16} a_{20} - {} \\
& \phantom{==} 699840000 a_{12} a_{16}^3 - 463050000 a_{20}^3.
\end{align*}
\end{theorem}

This seems to be the first explicit presentation of $H^*(R_7; \Q)$ in the literature.

\subsection{The proof for $R_7$}
Here we explain how we implement the method in Section \ref{subsec:genstratcomps} in the case of $R_7$. The ring $H^*(B\Spin(16); \Q)$ is given by 
$$
H^*(B\Spin(16); \Q) = \Q[p_1, p_2, p_3, p_4, e, p_5, p_6, p_7]
$$
where the $p_i=p_i(v_{16}) \in H^{4i}(B\Spin(16); \Q)$ are the Pontryagin classes of $v_{16}$, the $16$-dimensional real vector bundle over $B\Spin(16)$ and $e=e(v_{16}) \in H^{16}(B\Spin(16); \Q)$ is the Euler class of the same bundle. Following on from the previous section we calculate
\[
\ch(h^*(u)) = \ch(\lambda^2 v_{16}) + \ch(\Delta_{16}^+),
\]
up to and including the term $\ch_{30}$. We provide some details about this computation in Section \ref{subsec:calcs}. 

We know that $H^*(R_7; \Q)$ is generated by $4$ elements of degrees $8, 12, 16, 20$.  Since the map $\phi : H^*(B\Spin(16); \Q) \to H^*(R_7; \Q)$ is surjective, we can take
$$
a_8 = \phi(p_2), \quad a_{12} =\phi(p_3), \quad a_{20} = \phi(p_5)
$$
as the generators in degree 8, 12, and 20. Note that $\phi(p_1) = 0$ since $H^4(R_7; \Q)$ is zero. The ring $H^*(B\Spin(16); \Q)$ has two generators in degree $16$ so we have to make a choice of an element in $H^{16}(B\Spin(16); \Q)$ which maps to a generator in $H^{16}(R_7; \Q)$. Our calculation of $\ch(h^*(u))$ yields   
$$
\ch_8(h^*(u)) = \frac{1}{336}p_2^2 + \frac{1}{28}p_4 + \frac{1}{2} e \mod p_1.
$$
Here mod $p_1$ means up to an element of the form $xp_1$. This shows that we may choose $a_{16} = \phi(p_4)$ as our $16$-dimensional generator. Having made this choice, we immediately obtain an expression for $\phi(e)$ using $\phi(\ch_8(h^*(u))) = 0$:  
$$
\phi(e) =  -\frac{1}{168}a_8^2 - \frac{1}{14}a_{16}.
$$
The expression for $\ch_{12}(h^*(u))$ is a polynomial in $p_1, p_2, p_3, p_4, e, p_5, p_6$ (it does not involve $p_7$ since $p_7$ has degree $28$). At this stage we have formulae for $\phi(p_i)$, $i=1, \ldots, 5$ and $\phi(e)$ in terms of $a_4, \ldots, a_{20}$. This leads to the formula    
$$
\phi(p_6) = \frac{13}{1512} a_8^3 + \frac{3}{14} a_8a_{16} - \frac{1}{20} a_{12}^2. 
$$
Repeating this with $\ch_{14}(h^*(u))$ gives
$$ 
\phi(p_7) = 
\frac{1}{168} a_8^2 a_{12} - \frac{1}{12} a_8 a_{20} + \frac{1}{14} a_{12} a_{16}.
$$

We now have formulas for $\phi$ applied to each of the $8$ generators of $H^*(B\Spin(16); \Q)$ as polynomials in $a_4, a_{12}, a_{16}, a_{20}$. This allows us to factorise the homomorphism $\phi$ as 
$$
\begin{CD}
H^*(\Spin(16); \Q) @>{\theta}>> \Q[a_4, a_{12}, a_{16}, a_{20}] @>{q}>> H^*(R_7; \Q)
\end{CD}
$$
where $q$ is the quotient homomorphism, and $\theta$ is defined by the above formulas expressing $\phi(p_i)$ for $i = 1, \dots, 7$ and $\phi(e)$ in terms of $a_4, \ldots, a_{20}$. It follows that $\ker(q)$ is generated by the $4$ elements
$$
\theta(\ch_{18}(h^*(u))), \quad \theta(\ch_{20}(h^*(u))), \quad \theta(\ch_{24}(h^*(u))), \quad \theta(\ch_{30}(h^*(u))).
$$
This gives us four polynomials
\[ r_{36}, r_{40}, r_{48}, r_{60} \in \Q[a_8, a_{12}, a_{16}, a_{20}], \]
such that
\[ H^*(R_7) = \Q[a_8, a_{12}, a_{16}, a_{20}]/(r_{36}, r_{40}, r_{48}, r_{60}).\]
Calculating the expressions for these four polynomials gives the relations in Theorem \ref{thm:presR7}.

\subsection{Comments on the computation} \label{subsec:calcs}

The paper \cite{MR3903053}, in particular Section~3, is very helpful for doing these calculations. We illustrate our methods by outlining one method of computing
\[ \ch(h^*(u)) = \ch(\lambda^2 v_{16}) + \ch(\Delta_{16}^+) \]
in terms of the Pontryagin classes and Euler class of $v_{16}$ in $H^*(B\Spin(16); \Q)$.  This is the computation which leads to  the presentation of $H^*(R_7; \Q)$.

We begin with the computation of $\ch(\Delta_{16}^+)$.  The standard method is to use the splitting principle, as described in Jay Wood's paper \cite{MR1443417}, and then use the correspondence between symmetric functions and the Pontryagin classes and the Euler class to express $\ch(\Delta_{16}^+)$ in terms of the generators chosen above. The explicit expansion up to degree 16 is
\begin{align*}
\ch(\Delta_{16}^+) = {} & 128 + 16p_1 + \frac{1}{3} p_1^2 + \frac{4}{3} p_2 + \frac{1}{360} p_1^3 + \frac{1}{30} p_1 p_2 + \frac{2}{15} p_3 + \frac{1}{80640} p_1^4 + \frac{1}{3360} p_1^2 p_2 + {} \\ 
& \frac{1}{315} p_1 p_3 + \frac{1}{5040} p_2^2 + \frac{17}{1260} p_4 + \frac{1}{2} e + \cdots
\end{align*}

Again, the standard method for calculating $\ch(\lambda^2 v_{16})$ is to use the splitting principle to express $\ch_{k}(\lambda^2 v_{16})$ as a polynomial in the Pontryagin classes and the Euler class. The explicit expansion up to degree 16 is
\begin{align*}
\ch(\lambda^2 v_{16}) = {} & 120 + 14 p_1 + \frac{7}{6} p_1^2 - \frac{4}{3} p_2 + \frac{7}{180} p_1^3 - \frac{1}{30} p_1 p_2 - \frac{2}{15} p_3 + \frac{1}{1440} p_1^4 - \frac{1}{72} p_1 p_3 + {} \\ 
& \frac{1}{360} p_2^2 + \frac{1}{45} p_4 + \cdots
\end{align*}

\section{Computations in representation theory} \label{sec:reptheorycomps}

The proofs of Theorems \ref{thm:minmodgen}, \ref{thm:minmodgenSpineven} and Lemma \ref{lem:kernelsofhmaps} both require computations in representation theory which follow a standard operating procedure.

\begin{enumerate}
\item \label{item:1}
First we need detailed information about the $\Lambda$-ring $R(G)$ where $G$ is an appropriate Lie group.  Sometimes this is in the literature, but usually in the case of $E_6, E_7, E_8$ we use \Magma{} to get this information. The main issue is that if we start with a representation $u$ of $G$ and ask \Magma{} to compute $\lambda^k u$ it will give the answer as a decomposition into a sum of irreducible representations. However, we want to express it as a polynomial in a chosen set of generators. All computational programmes currently use the fundamental representations as the choice of generators, so that is our default choice. The labeling of weights for $G$ are different in various sources; we follow the Bourbaki convention \cite[Ch.\ VI, Planches I-IX]{MR0240238} for ordering the fundamental dominant weights, which is used in \Magma{}.  We now have to use \Magma{} to convert the decomposition of $\lambda^k u$ into a polynomial in the fundamental representations.  
\item
Next we have to soften this information to allow us to extract what we need.  The first step in this process is to pass from the $\Lambda$-ring $R(G)$  to the quotient $\Lambda$-ring 
$$
S(G) = R(G)/I^2(G).
$$
Our standard notation is $[x] \in S(G)$ for the image of the element $x$ in $R(G)$ under the quotient homomorphism $R(G) \to S(G)$.  If $x \in R(G)$ we write $\eps_x \in \Z$ for the augmentation of $x$.  If $x, y \in R(G)$ then 
$$
(x - \eps_x)(y - \eps_y) \in I^2(G)
$$
and it follows that in $S(G)$
$$
[xy] = \eps_y[x] + \eps_x[y] - \eps_x\eps_y.
$$ 
This will give the $\lambda$-operations in $S(G)$.  The only practical way to handle the arithmetic in this step of the process is to use \Magma.
\item
Both of the results we prove using this procedure involve the $\Psi$-module
$$
V(G) = S(G)/\Z.
$$
Note that $S(G)$ is a $\Lambda$-ring but $V(G)$ is not. However, both are $\Psi$-modules and the quotient map $S(G) \to V(G)$ is a map of $\Psi$-modules.
\end{enumerate}

\subsection{The proof of Theorem \ref{thm:minmodgen}}

Recall that in Theorem \ref{thm:minmodgen} we are assuming that $G$ is a simply connected Lie group not isomorphic to $\Spin(2n)$. The discussion in Section \ref{sec:Psimodule}, in particular the inductive formula for the Adams operations, show that the images of $[u], \psi^2[u], \ldots, \psi^n[u]$ generate $V(G) \otimes \Q$ if and only if the images of $[u], \lambda^2[u], \ldots, \lambda^n[u]$ generate $V(G) \otimes \Q$.  

For $\SU(n), \Sp(n)$ and $G_2$ the representation $u$ (of dimension $n$, $2n$ and $7$, respectively) actually generates $R(G)$ as a $\Lambda$-ring (see \cite[Theorems~7.4 and 7.6]{MR252560} and \cite[Section~22.3]{MR1153249}, respectively). The result is immediate. For the remaining cases we prove by direct computation that $u$ yields a generator of $V(G)\o\Q$.

For $G = \Spin(2n+1)$ the representation ring $R(G)$ is generated as a $\Lambda$-ring by the vector representation $v = v_{2n+1}$ (dimension $2n+1$) and the spin representation $\Delta = \Delta_{2n+1}$ (dimension $2^n$). Furthermore, the following equation holds in $R(G)$ (\cite[Theorem~10.3]{MR1249482}). 
\[ \Delta^2 = \lambda^n v + \lambda^{n-1} v + \cdots + v + 1. \]
In $S(G) \otimes \Q$ we have $[\Delta^2] = 2^{n+1}[\Delta] - 2^{2n}$. Therefore, the image of $[\Delta]$ in $V(G) \otimes \Q$ is a rational multiple of the image of $[\lambda^n v] + [\lambda^{n-1} v] + \cdots + [v]$. 

We illustrate how we use \Magma{} in the case $G = F_4$. The fundamental representations of $F_4$ are
\[ \rho_1, \rho_2, \rho_3, \rho_4, \quad \text{ of dimension} \quad 52,1274, 273, 26, \text{ respectively,}\]
and $u = \rho_4$. 
In $R(F_4)$ the $\lambda$-operations are given by
\[ \lambda^2 u = \rho_1 + \rho_3,\] 
\[ \lambda^3 u = \rho_1u + \rho_2 - u, \]
\[ \lambda^4 u = \rho_1^2 + \rho_1\rho_3 - \rho_2 - u^2.\]
This is the input for step \ref{item:1} of the standard operating procedure.  

Next we reduce these formulas to the $\Lambda$-ring $S(F_4)$. This gives the following formulas.
\[ \lambda^1[u] = [u],\]
\[\lambda^2[u] = [\rho_1] + [\rho_3],\]
\[\lambda^3[u]  = 26[\rho_1] +[\rho_2] +51[u] - 1352,\]
\[\lambda^4[u] = 377[\rho_1] - [\rho_2] + 52[\rho_3] - 52[u] - 16224.\]
Now we reduce to $V(F_4) = S(F_4)/\Z$.  This leads to the $4\times 4$ matrix.
\[
\begin{pmatrix}
0 & 1& 26 & 377  \\
0 & 0 & 1 & -1  \\
0 & 1 & 0 & 52\\
1 & 0 & 51 & -52 \\
\end{pmatrix}
\]
The columns of this matrix are the coordinates, in the basis defined by the fundamental representations $\rho_1, \rho_2, \rho_3, \rho_4 = u$, of the elements in $V(F_4)$ defined by $[\rho_4], \lambda^2[\rho_4], \lambda^3[\rho_4], \lambda^4[\rho_4]$. The determinant of this matrix is $351$. Therefore, $V(F_4)\o\Q$ is generated as a $\Psi$-module by the element defined by $u$.  

Implementing this procedure for $E_6$, $E_7$, and $E_8$ leads to the following three matrices, all of which are non-singular, as required.
\begin{enumerate}
\item
For $E_6$
$$
\begin{pmatrix}
1&    0&     0&    1& -702&-3483\\
0&    0&    0&  351& 3834&17496\\
0&    1&    0&  -27&  -77& -324\\
0&    0&    1&    0&  -54& -236\\
0&    0&    0&   79&  756& 2511\\
0&    0&    0& -405&-2759& 2754\\
\end{pmatrix}
$$
$$
\det =  -27208467.
$$
\item
For $E_7$
$$
\begin{pmatrix} 
0&   0&   0&   0&  -42504& -834648&-4655288\\
0&   0&   0&   0&8645&  207424& 2129083\\.
0&   0&   0&   0& 968&   20902&  166704\\
0&   0&   0&   1&   0&-267&   -2848\\
0&   0&   1&   0&-132&-856&   -5831\\
0&   1&   0&   1&  56&  -14742& -179760\\
1&   0&   1&   0&   -7371&  -79648& -560651\\
\end{pmatrix} 
$$
$$
\det =  -1997102661696.
$$
\item
For $E_8$
$$ 
\begin{pmatrix}   
0& 0& 0& 0& 0&-401581533&  -48147450080&-2346940420190\\
0& 0& 0& 0& 0&   6661497& 860320742&   49370806120\\
0& 0& 0& 0& 0&185628&  22371987&1105454390\\
0& 0& 0& 0& 1&-1& -7999&   -514878\\
0& 0& 0& 1& 0& -3627&   -112746&  -2002508\\  
0& 0& 1& 0&   248& 34255&  -9767140&-735062326\\  
0& 1& 0&   247&  -496&  -5059262&-205995340&   -5451187498\\  
1& 1&   495& 30380&   2573495& 304455619&   12658360729&  209067977980\\
\end{pmatrix} 
$$
$$
\det = 22804152835143344418390.
$$
\end{enumerate}

\subsection{The proof of Theorem \ref{thm:minmodgenSpineven}}

Again, everything we need is contained in \cite[Theorem~10.3]{MR1249482}. Let $G = \Spin(2n)$. The representation ring $R(G)$ is generated as a $\Lambda$-ring by the vector representation $v = v_{2n}$ (dimension $2n$) and the two spin representations $\Delta^+ = \Delta^+_{2n}$ and $\Delta^- = \Delta^-_{2n}$ (both have dimension $2^{n-1}$). This time, the following equation in $R(G)$ is sufficient for our purpose. 
\[ \Delta^+ \Delta^- = \lambda^{n-1} v + \lambda^{n-3} v + \cdots . \]
It follows that in $S(G) \otimes \Q$ we have $[\Delta^+ \Delta^-] = 2^{n-1}[\Delta^+] + 2^{n-1}[\Delta^-] - 2^{2n-2}$. Therefore, $V(G) \otimes \Q$ is generated by the images of $[v]$ and $[\Delta^+]$, as required.  

\subsection{The proof of Lemma \ref{lem:kernelsofhmaps}}

This time we start with the case of $E_6$. The fundamental representations of $E_6$ are
\[ \rho_1, \quad \rho_2 = a, \quad \rho_3 = \lambda^2 \rho_1, \quad \rho_4 = \lambda^3 \rho_1 = \lambda^3 \rho_6,\quad \rho_5 = \lambda^2\rho_6, \quad \rho_6\]
where $a$ is the $78$-dimensional adjoint representation and $\rho_1$, $\rho_6$ are the two inequivalent $27$-dimensional representations.

The fundamental representations of $\Spin(10)$ are
\[ v = v_{10}, \quad \lambda^2 v, \quad \lambda^3 v, \quad \Delta^+, \quad \Delta^-,\]
where $v = v_{10}$ is the $10$-dimensional vector representation and $\Delta^{\pm}$ are the two $16$-dimensional spin representations. 

Restricting the six fundamental $E_6$-representations to $\Spin(10)$ yields the following formulas, as in \cite[Lemma~2]{MR0440584}. 
\begin{enumerate}
\item
$h_6^*(\rho_1) =  v +  \Delta^- +1$,
\item 
$h_6^*(\rho_2) = \lambda^2 v + \Delta^+ + \Delta^- + 1$,
\item
$h_6^*(\rho_3) = v + \lambda^2 v + \lambda^3 v +  \Delta^- + v \Delta^-$,
\item
$h_6^*(\rho_4) = \lambda^2 v + 2\lambda^3 v + (\lambda^2 v) \Delta^+ + (\lambda^2 v) \Delta^- + v (\lambda^3 v)$,
\item
$h_6^*(\rho_5) = v + \lambda^2 v + \lambda^3 v +  \Delta^+ + v \Delta^+$,
\item
$h_6^*(\rho_6) =  v +  \Delta^+ +1$.
\end{enumerate}

Next we reduce these formulas to the $\Lambda$-ring $S(G)$.  This produces the following six equations in $S(\Spin(10))$.
\begin{enumerate}
\item
$h_6^*[\rho_1] = [v] +  [\Delta^-]$ + 1,
\item 
$h_6^*[\rho_2] = [\lambda^2 v] + [\Delta^+] + [\Delta^-] $ + 1,
\item
$h_6^*[\rho_3] = 17 [v] + [\lambda^2 v] + [\lambda^3 v] +  11[\Delta^-] - 160$,
\item
$h_6^*[\rho_4] = 120[v] + 33[\lambda^2 v] + 12[\lambda^3 v] + 45[\Delta^-]  + 45[\Delta^+] - 2640$,
\item
$h_6^*[\rho_5] = 17[v] + [\lambda^2 v] + [\lambda^3 v] +  11[\Delta^+] -160$,
\item
$h_6^*[\rho_6] = [v] +  [\Delta^+] +1 $.
\end{enumerate}

This gives the following matrix with 5 rows and 6 columns which is the matrix of $(h_6)_* : V(E_6) \to V(\Spin(10))$ in the chosen bases for $V(E_6)$ and $V(\Spin(10))$.
\[
\begin{pmatrix}
1 &  0 & 17 & 120 & 17 & 1 \\
0 &  1 &  1 & 33  & 1  & 0 \\
0 &  0 &  1 & 12  & 1  & 0 \\
0 &  1 &  0 & 45  & 11 & 1 \\
1 &  1 & 11 & 45  & 0  & 0 
\end{pmatrix}
\]
The rank of this matrix is $4$ and the nullity is $2$, thus the transpose has rank $4$ and nullity $1$. Therefore, the kernel of $h_6^*: V^*(\Spin(10))\o\Q \to V^*(E_6)\o\Q$ is one dimensional, proving part (i) of Lemma \ref{lem:kernelsofhmaps}. 

The same procedure gives the corresponding matrices for the  homomorphisms $h_7: \Spin(12) \to E_7$ and $h_8: \Spin(16) \to E_8$.  However, in this case we have to use \Magma{} to calculate $\Lambda$-operations and also to soften this information to give the linear maps $h_7^*: V(E_7)\o\Q \to V(\Spin(12))\o\Q$ and $h_8^*: V(E_8)\o\Q \to V(\Spin(16))\o\Q$. 

Using the standard bases defined by fundamental representations, we find that the matrices of $h_7^*$ and $h_8^*$ are, respectively,  
\[ \begin{pmatrix}   
    0&   34&  220&25840& 1890&   88&    2\\
    1&    0&   66&  890&   56&    2&    0\\
    0&    2&   12& 1320&   34&    0&    0\\
    0&    0&    1&  145&   24&    1&    0\\
    2&   12&  200& 9120&  220&    0&    0\\
    0&    2&    0&  120&  210&   24&    1\\
\end{pmatrix}\]
and
\[
\begin{pmatrix}
       160&      9056&    597840&1005303376&  13865488&    125552&       560&         0\\
        -1&      -259&    -15473& -49227299&   -344942&      4241&       128&         1\\
         0&       144&      4480&  17099280&    296400&      3168&        16&         0\\
         1&        -3&      1808&  -1025376&    -23023&      -240&        -1&         0\\
         0&        16&       560&   2093696&     22048&       112&         0&         0\\
         0&        -2&      -105&   -370711&      3367&       121&         1&         0\\
        16&       816&     58608&  94909584&    920192&      4368&         0&         0\\
        -1&        15&     -5568& -12391764&    -22477&      4944&       119&         1\\

\end{pmatrix}.\]

In both cases the rank is $4$.  Taking account of the fact that the first is a $6 \times 7$ matrix and the second is an $8 \times 8$ matrix, this shows that the nullity of the transpose of the first is $2$ and for the second it is $4$. This completes the proof.   

\bibliographystyle{alpha}
\bibliography{JRT}

\end{document}